\documentclass[10pt]{article}
 \usepackage{amsmath, amsfonts, amsthm, amssymb, amscd, enumerate, url, slashed, stmaryrd, upgreek, thmtools, mathdots, mathtools}
 \usepackage{bm}
 \usepackage{float}
 \usepackage{tikz}
 \usetikzlibrary{matrix}
 
\begingroup 
 \makeatletter
 \@for\theoremstyle:=definition,remark,plain\do{%
 \expandafter\g@addto@macro\csname th@\theoremstyle\endcsname{%
 \addtolength\thm@preskip\parskip
 }%
 }
 \endgroup
 
\usepackage[titletoc,toc,title]{appendix}

 \usepackage[colorlinks = true,
 linkcolor = blue,
 urlcolor = cyan,
 citecolor = red,
 anchorcolor = blue,
 pagebackref]{hyperref}
 
\usepackage[alphabetic,backrefs]{amsrefs}
 
 \usepackage[parfill]{parskip}
 \usepackage{anysize}
 \marginsize{1in}{1in}{1in}{1in}

\declaretheorem[name=Theorem,numberwithin=section]{thm}
\declaretheorem[name=Proposition,numberlike=thm]{prop}
\declaretheorem[name=Lemma,numberlike=thm]{lemma}
\declaretheorem[name=Corollary,numberlike=thm]{cor}

\declaretheorem[name=Definition,style=definition,qed=$\blacktriangle$,numberlike=thm]{defn}
\declaretheorem[name=Example,style=definition,qed=$\blacktriangle$,numberlike=thm]{ex}
\declaretheorem[name=Remark,style=definition,qed=$\blacktriangle$,numberlike=thm]{rmk}

\newcounter{noteCounter}
\DeclareMathOperator\Div{div}

\DeclareMathOperator\tr{tr}

\DeclareMathOperator{\im}{im}

\DeclareMathOperator{\rank}{rank}

\newcommand{\eps}{\varepsilon}
\newcommand{\spi}{\slashed S}

\newcommand{\vol}{\mathsf{vol}}

\newcommand{\cT}{\mathcal T}
\newcommand{\cV}{\mathcal V}

\newcommand{\real}{\operatorname{\mathrm{Re}}}

\newcommand{\G}{\mathrm{G}_2}
\newcommand{\Spin}[1]{\mathrm{Spin}(#1)}

\newcommand{\N}{\mathbb N}
\newcommand{\R}{\mathbb R}
\newcommand{\C}{\mathbb C}

\newcommand{\PR}{\mathbb P}

\newcommand{\ph}{\varphi}
\newcommand{\ps}{\psi}
\newcommand{\Ph}{\Phi}

\newcommand{\hk}{\mathbin{\! \hbox{\vrule height0.3pt width5pt depth 0.2pt \vrule height5pt width0.4pt depth 0.2pt}}}

\newcommand{\rest}[2]{{#1}|_{#2}}

\newcommand{\wt}{\widetilde}

\begin{document}

\title{A special class of $k$-harmonic maps \\ inducing calibrated fibrations}

\author{Anton Iliashenko \\ \emph{Department of Pure Mathematics, University of Waterloo} \\ \tt{ailiashenko@uwaterloo.ca} \and Spiro Karigiannis \\ \emph{Department of Pure Mathematics, University of Waterloo} \\ \tt{karigiannis@uwaterloo.ca}}

\maketitle

\begin{abstract}
We consider two special classes of $k$-harmonic maps between Riemannian manifolds which are related to calibrated geometry, satisfying a first order fully nonlinear PDE. The first is a special type of weakly conformal map $u \colon (L^k, g) \to (M^n, h)$ where $k \leq n$ and $\alpha$ is a calibration $k$-form on $M$. Away from the critical set, the image is an $\alpha$-calibrated submanifold of $M$. These were previously studied by Cheng--Karigiannis--Madnick when $\alpha$ was associated to a vector cross product, but we clarify that such a restriction is unnecessary. The second, which is new, is a special type of weakly horizontally conformal map $u \colon (M^n, h) \to (L^k, g)$ where $n \geq k$ and $\alpha$ is a calibration $(n-k)$-form on $M$. Away from the critical set, the fibres $u^{-1} \{ u(x) \}$ are $\alpha$-calibrated submanifolds of $M$.

We also review some previously established analytic results for the first class; we exhibit some explicit noncompact examples of the second class, where $(M, h)$ are the Bryant--Salamon manifolds with exceptional holonomy; we remark on the relevance of this new PDE to the Strominger--Yau--Zaslow conjecture for mirror symmetry in terms of special Lagrangian fibrations and to the $\G$ version by Gukov--Yau--Zaslow in terms of coassociative fibrations; and we present several open questions for future study.
\end{abstract}

\tableofcontents

\section{Introduction} \label{sec:intro}

The natural partial differential equations which arise in Riemannian geometry are usually second order. Some important examples are:
\begin{enumerate}[(i)] \setlength\itemsep{-1mm}
\item an Einstein metric [$\mathrm{Ric}_g = \lambda g$, where $\mathrm{Ric}$ is the Ricci curvature]
\item a minimal submanifold [$H = 0$, where $H$ is the mean curvature]
\item a Yang--Mills connection $\nabla$ on a vector bundle [$(d^{\nabla})^* F^{\nabla} = 0$, where $F^{\nabla}$ is the curvature]
\item a $k$-harmonic map $u \colon (M_1, g_1) \to (M_2, g_2)$ between Riemannian manifolds [$\Div (|du|^{k-2} du) = 0$]
\end{enumerate}
All of the above geometric objects are also \emph{variational}. That is, the PDEs are Euler--Lagrange equations for some natural geometric functional or ``energy'', and hence such objects are \emph{critical points} of these functionals, but may not in general be (local) minima.

A common feature is that when there is additional geometric structure present, one can identify a natural \emph{special class of solutions} which:
\begin{itemize} \setlength\itemsep{-1mm}
\item satisfy a (usually fully nonlinear) \emph{first order} PDE, and
\item are actually global minimizers of the functional within a particular class of variations.
\end{itemize}
With respect to the particular examples above, these special first order solutions are:
\begin{enumerate}[(i)] \setlength\itemsep{-1mm}
\item a \emph{special holonomy metric}: Calabi--Yau, hyperk\"ahler, quaternionic-K\"ahler, $\G$, or $\Spin{7}$. These are all Einstein, and most are Ricci-flat. [The condition of special holonomy is first order on the metric in each case, but there does not seem to be any unified way of describing these, and it is unknown if they are global minimizers of the Einstein--Hilbert functional within some particular class of variations.]
\item a \emph{calibrated submanifold} of a special holonomy manifold. These are all minimal. The calibrated condition is a first order condition on the immersion. They are global minimizers of the volume functional in a given homology class.
\item an \emph{instanton} on a vector bundle over a special holonomy manifold. These are all Yang--Mills. The instanton condition is a first order condition on the connection, being an algebraic condition on the curvature. In many cases, a characteristic class argument shows that they are global minimizers of the Yang--Mills energy.
\end{enumerate}
Note that all the special first order solutions in (i), (ii), and (iii) described above are related to Riemannian manifolds with special holonomy. [This is not necessary. Classical self-dual and anti-self-dual instantons are special Yang--Mills connections on a Riemannian $4$-manifold, with no special holonomy.]

In this paper, we discuss two classes of special first order solutions to (iv) above, called \emph{Smith maps}. They are special types of $k$-harmonic maps $u \colon (M_1, g_1) \to (M_2, g_2)$ between pairs of Riemannian manifolds, which are intimately related to both \emph{calibrated geometry} and \emph{conformal geometry}:
\begin{itemize}
\item For $u \colon (L^k, g) \to (M^n, h)$, with $k \leq n$ and $\alpha \in \Omega^k (M)$ a closed calibration, we define a \emph{Smith immersion}, which is a special type of weakly conformal $k$-harmonic map. If $L^0$ is the open subset on which $du \neq 0$, then $u \colon L^0 \to M$ is an immersion, whose image $u(L^0)$ is $k$-dimensional $\alpha$-calibrated submanifold of $(M, h)$. Moreover, the notion of Smith immersion is invariant under \emph{conformal change} of the domain metric $g$. Conversely, if $u \colon (L^k, g) \to (M^n, h)$ is a weakly conformal $k$-harmonic map such that $u(L^0)$ is $\alpha$-calibrated, then $u$ is a Smith immersion. (Theorem~\ref{thm:immersion-ineq}.)
\item For $u \colon (M^n, h) \to (L^k, g)$, with $n \geq k$ and $\alpha \in \Omega^{n-k}(M)$ a closed calibration, we define a \emph{Smith submersion}, which is a special type of weakly horizontally conformal $k$-harmonic map. If $M^0$ is the open subset on which $du \neq 0$, then the fibres $u^{-1} \{ u(x) \}$ of $u \colon M^0 \to L$ are $(n-k)$-dimensional $\alpha$-calibrated submanifolds of $(M, h)$. Moreover, the notion of Smith submersion is invariant under \emph{horizontally conformal change} of the domain metric $h$. Conversely, if $u \colon (M^n, h) \to (L^k, g)$ is a weakly horizontally conformal $k$-harmonic map such that the fibres of $u|_{M^0}$ are $\alpha$-calibrated, then $u$ is a Smith submersion. (Theorem~\ref{thm:submersion-ineq}.)
\end{itemize}
The notion of Smith immersions was previously studied by Cheng--Karigiannis--Madnick in~\cite{CKM1} and~\cite[Section 3.3]{CKM2}, inspired by an unpublished preprint of Smith~\cite{Smith}. We review it here, and clarify that it extends from calibrations associated to vector cross products to any calibrations. (This was implicit in~\cite[Section 3.3]{CKM2}.) The notion of Smith submersions is \emph{new} in the present paper.

In each case, we establish a fundamental pointwise inequality in Theorems~\ref{thm:immersion-ineq} and~\ref{thm:submersion-ineq}, respectively, which itself is obtained by combining the fundamental inequality of calibrated geometry and the Hadamard inequality. We then use these pointwise inequalities, together with the assumption that $d \alpha = 0$, to prove the associated integral \emph{energy inequalities} in Theorems~\ref{thm:Smith-immersion-energy-ineq} and~\ref{thm:Smith-submersion-energy-ineq}, respectively, when the domain is compact. This immediately yields the $k$-harmonicity of such maps. We also give direct proofs of $k$-harmonicity by differentiating the Smith equations, which also explicitly show the importance of the $d \alpha = 0$ assumption.

The two constructions should also be viewed as special first order versions of the following particular classical results~\cite[(3.5) and (3.10)]{Urakawa} from harmonic map theory:
\begin{itemize} \setlength\itemsep{-1mm}
\item a Riemannian immersion $u \colon (L, g) \to (M, h)$ is harmonic $\iff$ the image is minimal,
\item a Riemannian submersion $u \colon (M, h) \to (L, g)$ is harmonic $\iff$ the fibres are minimal.
\end{itemize}

In the final section, we briefly discuss the analytic results for Smith immersions which were established in~\cite{CKM1}, discuss several explicit examples of Smith submersions with noncompact domains, comment on the relevance of Smith submersions to the SYZ and GYZ ``conjectures'' involving special Lagrangian and coassociative fibrations, and collect several open questions for future study.

\subsubsection*{Conventions and notation.} All manifolds are \emph{oriented} Riemannian manifolds, though not necessarily compact. As usual a superscript on a manifold such as $M^n$ means $\dim M = n$. All maps between manifolds are smooth.

We often use the Riemannian metric (via the musical isomorphism) to identify vector fields and $1$-forms, and more generally tensors of mixed type with covariant tensors. We use $\cT^m$ for the space of smooth $m$-tensors (that is, smooth sections of the $m^{\text{th}}$ tensor power of the cotangent bundle), we use $\Omega^p$ for the space of $p$-forms, $\star$ for the Hodge star operator, and $\vol$ for the Riemannian volume form. We use $\nabla$ for the Levi-Civita connection. We write $\Div \colon \cT^m \to \cT^{m-1}$ for the Riemannian \emph{divergence}, given in terms of a local orthonormal frame by $(\Div A)_{j_1 \cdots j_{m-1}} = \nabla_i A_{i j_1 \cdots j_{m-1}}$. (We sum over repeated indices.)

For us, a \emph{calibration} $\alpha$ is a comass one differential form, not necessarily closed. (Some authors call this a semi-calibration or pre-calibration.) When $\alpha$ is also closed, we call it a \emph{closed calibration}.

The following result is a version of \emph{Hadamard's inequality} that we use frequently.
\begin{prop}[Hadamard's inequality] \label{prop:Hadamard}
Let $A \colon (V_1^{n_1}, g_1) \to (V_2^{n_2}, g_2)$ be a linear map between real inner product spaces where $n_k = \dim V_k$. Then $| \Lambda^{n_1} A | \leq \frac{1}{(\sqrt{n_1})^{n_1}} |A|^{n_1}$ with equality if and only if $A^* g_2 = \lambda^2 g_1$ with $\lambda^2 = \frac{1}{n_1} |A|^2$.
\end{prop}
\begin{proof}
A proof can be found, for example, in~\cite[Corollary 2.5 and Lemma 2.1]{CKM1}.
\end{proof}

\textbf{Funding.} The research of the second author is supported by an NSERC Discovery Grant. 

\textbf{Acknowledgements.} This work comprises one part of the first author's PhD thesis. The authors thank Da Rong Cheng and Jesse Madnick for useful discussions, and are very grateful to the referee for numerous helpful comments and suggestions that greatly improved our earlier draft.

\section{Preliminaries} \label{sec:prelim}

In this section we review some standard material on calibrations and $p$-harmonic maps.

\subsection{Calibrations} \label{sec:calibrations}

The classical theory of calibrated geometry was initiated by Harvey--Lawson~\cite{HL}. A good reference for beginners is the text of Joyce~\cite{Joyce-book}. Let $(M^n, h)$ be a Riemannian manifold.

\begin{defn} \label{defn:calibration}
Let $\alpha \in \Omega^k$ on $(M^n, h)$. We say that $\alpha$ is a \emph{calibration} if
\begin{equation} \label{eq:calibration}
\alpha (v_1 \wedge \cdots \wedge v_k) \leq | v_1 \wedge \cdots \wedge v_k | \quad \text{for all $v_1, \ldots, v_k \in T_x M$ and all $x \in M$}.
\end{equation}
This is clearly equivalent to saying that
$$ -1 \leq \alpha(e_1, \ldots, e_k) \leq 1 \quad \text{for all \emph{orthonormal} $e_1, \ldots, e_k \in T_x M$ and all $x \in M$.} $$

Let $L^k$ be an oriented submanifold of $M$. We say $L$ is \emph{calibrated} with respect to $\alpha$ if $\alpha|_{L} = \vol_L$, where $\vol_L$ is the Riemannian volume form associated to the orientation and the induced metric $h|_L$. (That is, $L$ is $\alpha$-calibrated if equality in~\eqref{eq:calibration} is attained on each oriented tangent space $T_x L$ of $L$.)
\end{defn}

The classical \emph{fundamental theorem of calibrated geometry} of Harvey--Lawson~\cite{HL} says that if the calibration form $\alpha$ is \emph{closed}, then a calibrated submanifold is locally volume minimizing in its homology class. In particular, if $d\alpha = 0$, then a calibrated submanifold is \emph{minimal} (has vanishing mean curvature).

We collect here some results and definitions on calibrations which are needed later.

\begin{lemma}[The first cousin principle] \label{lemma:cousin}
Let $\alpha \in \Omega^k$ be a calibration, and let $L_x \in \Lambda^k (T_x M)$ be an oriented $k$-dimensional subspace which is calibrated with respect to $\alpha$. If $e_1, \ldots e_{k-1}$ are orthonormal in $L_x$ and $w \in L_x^{\perp}$, then $\alpha(e_1, \ldots, e_{k-1}, w) = 0$.
\end{lemma}
\begin{proof}
We can choose $e_k \in L_x$ so that $e_1, \ldots, e_k$ is an oriented orthonormal basis of $L_x$. Let $w_t = (\cos t) e_k + (\sin t) w$. Then $e_1, \ldots, e_{k-1}, w_t$ are orthonormal for all $t \in \R$. Thus we have that
\begin{align*}
f(t) := \alpha(e_1, \ldots, e_{k-1}, w_t) & = (\cos t) \alpha(e_1, \ldots, e_{k-1}, e_k) + (\sin t) \alpha(e_1, \ldots, e_{k-1}, w)
\end{align*}
satisfies $f(t) \leq 1$ for all $t \in \R$ with equality at $t=0$. Thus $f'(0) = \alpha(e_1, \ldots, e_{k-1}, w) = 0$.
\end{proof}

\begin{prop} \label{prop:star-calibration}
If $\alpha \in \Omega^k$ is calibration, then $\star \alpha \in \Omega^{n-k}$ is also a calibration.
\end{prop}
\begin{proof}
Using the metric we can identify $\Lambda^k (T_x M)$ with $\Lambda^k (T_x^* M)$. Let $\Pi_x = e_1 \wedge \cdots \wedge e_k$, where $e_1, \ldots, e_k$ are orthonormal. Then using the fact that $\star$ is an isometry, we have
$$(\star \alpha) (\Pi_x) = g( \star \alpha, \Pi_x ) = g( \star^2 \alpha, \star \Pi_x ) = \pm g( \alpha, \star \Pi_x ) = \pm \alpha(\star \Pi_x) \in [-1,1], $$
because $\alpha$ is a calibration.
\end{proof}

\begin{defn} \label{defn:P-from-alpha}
Let $\alpha \in \Omega^{k}$. Define $P_{\alpha} \colon \Gamma(\Lambda^{k-1}(TM)) \to \Gamma(TM)$ by
$$ g ( P_{\alpha} (v_1 \wedge \cdots \wedge v_{k-1}), v_k) = \alpha (v_1 \wedge \cdots \wedge v_k). $$
That is, $P_{\alpha}$ is the vector-valued $(k-1)$-form obtained by ``raising an index'' on $\alpha$ using the metric.
\end{defn}

\begin{rmk} \label{rmk:VCP}
For some calibrations $\alpha$, the vector-valued form $P_{\alpha}$ is a \emph{vector cross product}. This means that $|P_{\alpha} (v_1 \wedge \cdots \wedge v_{k-1})|^2 = |v_1 \wedge \cdots \wedge v_{k-1}|^2$. This holds, in particular, for the K\"ahler calibration of degree 2, and for the associative and Cayley calibrations. See~\cite[Section 2]{CKM1} for more details. One of the key points of our Section~\ref{sec:Smith-immersions} below is the observation that the results of~\cite{CKM1} continue to hold for all calibrations, not just for those for which $P_{\alpha}$ is a vector cross product.
\end{rmk}

\begin{prop} \label{prop:P-adjoint}
Let $\alpha \in \Omega^{k}$. The adjoint $P_{\alpha}^{\top} \colon \Gamma(TM) \to \Gamma(\Lambda^{k-1}(TM))$ is given by
$$ P_{\alpha}^{\top}(v) = (-1)^{k-1} v \hk \alpha. $$
(There is a metric identification here of $\Lambda^{k-1} (TM)$ and $\Lambda^{k-1} (T^* M)$.)
\end{prop}
\begin{proof}
Let $v_1, \ldots, v_k \in \Gamma(TM)$. We compute
\begin{align*}
g (P_{\alpha} (v_1 \wedge \cdots \wedge v_{k-1}), v_k) & = \alpha (v_1 \wedge \cdots \wedge v_k) \\
& = g(v_1 \wedge \cdots \wedge v_k, \alpha) \\
& = (-1)^{k-1} g(v_k \wedge v_1 \wedge \cdots \wedge v_{k-1}, \alpha) \\
& = (-1)^{k-1} g(v_1 \wedge \cdots \wedge v_{k-1}, v_k \hk \alpha),
\end{align*}
hence the result follows.
\end{proof}

\subsection{Harmonic maps and $p$-harmonic maps} \label{sec:p-harmonic}

We briefly review some basic facts about harmonic maps and $p$-harmonic maps. For more details, the reader can consult Eells--Lemaire~\cite{EL} or Baird--Gudmundsson~\cite{BG}.

If $u \colon (M_1^{n_1}, g_1) \to (M_2^{n_2}, g_2)$ is a smooth map between Riemannian manifolds, then its differential $du$ is a smooth section of $T^* M_1 \otimes u^* T 
M_2$, and its value at $x \in M_1$ is the linear map $du_x \colon T_x M_1 \to T_{u(x)} M_2$. The bundle $T^* M_1 \otimes u^* T M_2$ has a natural fibre metric $g_1^{-1} \otimes u^* g_2$ which allows us to define the smooth function $|du|^2$ on $M_1$. One can also verify that 
\begin{equation} \label{eq:du-norm}
|du|^2 = \tr_{g_1} (u^* g_2).
\end{equation}
A useful observation is that if $e_1, \ldots, e_{n_1}$ is a local orthonormal frame for $(M_1, g_1)$, then
\begin{equation} \label{eq:du-norm-v2}
|du_x|^2 = \sum_{i=1}^{n_1} (u^* g_2)_x (e_i, e_i) = \sum_{i=1}^{n_1} g_2 (du_x (e_i), du_x (e_i)).
\end{equation}

\begin{defn} \label{defn:p-energy}
Let $u \colon (M_1, g_1) \to (M_2, g_2)$ be a smooth map. Let $p \in [2, \infty)$. If $M_1$ is compact, then the \emph{$p$-energy} of $u$ is defined to be
$$ E_p (u) \coloneqq \frac{1}{(\sqrt{p})^p} \int_{M_1} |du|^p \vol_{M_1}. $$
Note that up to a constant factor (which is chosen for later convenience), the $p$-energy is the $p^{\text{th}}$ power of the $L^p$ norm of $du$. We say that a map $u$ is \emph{$p$-harmonic} if it is a critical point of the functional $E_p$. That is, a $p$-harmonic map is a solution to the Euler--Lagrange equation for the $p$-energy functional. This equation is
\begin{equation} \label{eq:p-harmonic}
\Div(|du|^{p-2} du) = 0 \in \Gamma(u^ *TM_2),
\end{equation}
and is called the \emph{$p$-harmonic map equation}. When $p=2$, this reduces to the classical elliptic harmonic map equation $\Div (du) = 0$, and a $2$-harmonic map is just called a harmonic map. But for $p > 2$ this equation is a \emph{degenerate} elliptic equation.

More generally, the section of $u^* TM_2$ given by
\begin{equation} \label{eq:p-tension}
\tau_p (u) \coloneqq \Div(|du|^{p-2} du)
\end{equation}
is called the \emph{$p$-tension} of $u$, so a map $u$ is $p$-harmonic if and only if it has vanishing $p$-tension. In fact, the $p$-tension $\tau_p (u)$ is, up to a positive factor, the negative gradient of the $p$-energy functional with respect to the $L^2$ inner product.

Note that if $M_1$ is not compact we can still take equation~\eqref{eq:p-harmonic} as the definition of $p$-harmonic.
\end{defn}

The $p$-energy and $p$-harmonic map equation are related to conformal geometry as follows. Let $f$ be a positive function on $M_1$, so $\tilde g_1 = f^2 g_1$ is another metric on $M_1$ in the same conformal class as $g_1$. Then we have
$$ |du|^2_{\tilde g_1, g_2} = f^{-2} |du|^2_{g_1, g_2} \qquad \text{and} \qquad \vol_{M_1, \tilde g_1} = f^{n_1}\vol_{M_1, g_1}. $$
It follows that
$$ E_{p, \tilde g_1, g_2} (u) = \frac{1}{(\sqrt{p})^p} \int_{M_1} |du|^p_{\tilde g_1, g_2} \vol_{M_1, \tilde g_1} = \frac{1}{(\sqrt{p})^p} \int_{M_1} f^{n_1 - p} |du|^p_{g_1, g_2} \vol_{M_1, g_1}, $$
and thus the $p$-energy of a map $u \colon (M_1^{n_1}, g_1) \to (M_2^{n_2}, g_2)$ is \emph{conformally invariant} (that is, depends only on the conformal class of $g_1$) if $p = n_1$. With a bit more effort, one can similarly compute that
$$ \tau_{p, \tilde g_1, g_2} (u) = f^{-p} \tau_{p, g_1, g_2} (u) + f^{-p-1} |du|^{p-2} (n_1 - p) g_1 (df, du), $$
which again shows that the notion of a $p$-harmonic map depends only on the conformal class of $g_1$ if $p = n_1$.

The case that has received the most attention classically is the conformal invariance of the $2$-energy (also called the Dirichlet energy) from a $2$-dimensional oriented Riemannian manifold $(\Sigma^2, g)$ into another Riemannian manifold $(M, h)$. Since this depends only on the conformal class of $g$ on $\Sigma^2$, we see that the notion of a harmonic map from a \emph{Riemann surface} $\Sigma^2$ into a Riemannian manifold is well-defined.

See Remarks~\ref{rmk:immersion-conf-inv} and Remarks~\ref{rmk:submersion-conf-inv} for the precise formulation of ``conformal invariance'' for Smith immersions and Smith submersions.

\section{Smith immersions} \label{sec:Smith-immersions}

The notion of a \emph{Smith immersion} was studied by Cheng--Karigiannis--Madnick in~\cite{CKM1} and~\cite[Section 3.3]{CKM2} where it was assumed that the calibration form $\alpha$ is induced from a vector cross product. In this section we introduce a slightly modified definition of Smith immersions which applies to \emph{any} calibration $\alpha$, not just those induced by vector cross products. In the vector cross product case, our new definition is equivalent to the earlier definition. Moreover, our more general definition still enjoys all the analytic properties established in~\cite[Sections 4 and 5]{CKM1}. See Section~\ref{sec:immersions-analysis}.

In this section, $u \colon (L^k, g) \to (M^n, h)$ is a smooth map between Riemannian manifolds, with $k \leq n$. Recall that $u \colon (L^k, g) \to (M^n, h)$ is an \emph{immersion} if $\rank(du_x) = k$ for all $x \in L$. 

\subsection{Smith immersions and the energy inequality} \label{sec:Smith-imm-energy}

Before we can define Smith immersions, we recall some facts about (weakly) conformal maps.
 
\begin{defn} \label{defn:weakly-conformal-immersion}
A smooth map $u \colon (L^k, g) \to (M^n, h)$ is called \emph{(weakly) conformal} if
$$ u^* h = \lambda^2 g $$
for some smooth function $\lambda \geq 0$ which is continuous (and smooth away from $0$) on $L$. This function $\lambda$ is called the dilation. It then follows from~\eqref{eq:du-norm} that necessarily $\lambda^2 = \frac{1}{k} |du|^2$.

Let $L^{0} \subseteq L$ be the open set where $|du| \neq 0$. From $u^* h = \frac{1}{k} |du|^2 g$, we deduce that $\rest{u}{L^0} \colon L^0 \to M$ is an immersion. When $L^0 = L$, we say that $u$ is a \emph{conformal immersion}. An immersion $u \colon (L^k, g) \to (M^n, h)$ is called a \emph{Riemannian immersion} if $u^* h = g$ on $L$, or equivalently if it is a conformal immersion with dilation $\lambda = 1$.
\end{defn}

\begin{thm} \label{thm:immersion-ineq}
Let $u \colon (L^k, g) \to (M^n, h)$ be a smooth map. Let $\alpha \in \Omega^{k} (M)$ be a calibration. Then
\begin{equation} \label{eq:immersion-ineq}
u^* \alpha \leq \lambda^k \vol_L, \quad \text{where $\lambda = \tfrac{1}{\sqrt{k}} |du|$}.
\end{equation}
Moreover, equality holds if and only if:
\begin{itemize} \setlength\itemsep{-1mm}
\item $u^* h = \lambda^2 g$ (so $u$ is a weakly conformal immersion), and
\item the image $u(L^0)$ is calibrated with respect to $\alpha$.
\end{itemize}
\end{thm}
\begin{proof}
We trivially have equality at points where $du$ is zero. Let $x \in L^0$. Let $e_1, \ldots, e_k$ be an orthonormal frame for $T_x L$. Then we have
\begin{align*}
(u^* \alpha) (e_1 \wedge \cdots \wedge e_k) & = \alpha ((\Lambda^k du) (e_1 \wedge \cdots \wedge e_k)) & & \\
& \leq |(\Lambda^k du) (e_1 \wedge \cdots \wedge e_k)| & & \text{(because $\alpha$ is a calibration)} \\
& = |\Lambda^k du| \, |e_1 \wedge \cdots \wedge e_k| & & \\
& \leq \lambda^k & & \text{(by Proposition~\ref{prop:Hadamard})},
\end{align*}
which concludes the proof of~\eqref{eq:immersion-ineq}.

Equality holds if and only if equality holds in the two inequalities of the above computation. If the second inequality above is an equality, then by Proposition~\ref{prop:Hadamard} we have $u^* h = \lambda^2 g$, so $u$ is weakly conformal. Let $x \in L^0$ and let $e_1, \ldots, e_k$ be an orthonormal frame for $T_x L$, so $\frac{1}{\lambda} du(e_1), \ldots, \frac{1}{\lambda} du(e_k)$ is an orthonormal frame for $du(T_x L) \subseteq T_{u(x)} M$. If the first inequality above is an equality, then we see that we must have $\alpha (\frac{1}{\lambda} du (e_1) \wedge \cdots \wedge \frac{1}{\lambda} du (e_k)) = 1$. That is, the image $u(L^0)$ is calibrated with respect to $\alpha$.
\end{proof}

\begin{defn} \label{defn:Smith-immersion}
If equality holds in~\eqref{eq:immersion-ineq}, we say that $u$ is a \textbf{Smith immersion} with respect to $\alpha$. That is, a Smith immersion with respect to $\alpha$ is a smooth map $u \colon (L^k, g) \to (M^n, h)$ such that
\begin{equation} \label{eq:Smith-immersion-eq}
u^* \alpha = \frac{1}{(\sqrt{k})^k} |du|^k \vol_L, \qquad u^* h = \frac{1}{k} |du|^2 g,
\end{equation}
at all points on $L$. [However, recall that the first equation automatically implies the second equation.] Note that, strictly speaking, a Smith immersion is only actually an immersion on the open subset $L^0 = \{ x \in L : du_x \neq 0 \}$ of $L$.
\end{defn}

\begin{thm}[Energy Inequality] \label{thm:Smith-immersion-energy-ineq}
Let $\alpha \in \Omega^{k}(M)$ be a \emph{closed} calibration. Let $u \colon (L^k, g) \to (M^n, h)$ be a Smith immersion with respect to $\alpha$. Suppose $L$ is compact. Then $u$ is $k$-harmonic in the sense that it is a critical point of $E_k$.
\end{thm}
\begin{proof}
For any smooth map $u \colon (L^k, g) \to (M^n, h)$, let $\lambda = \frac{1}{\sqrt{k}} |du|$. Using~\eqref{eq:immersion-ineq} we have
\begin{align*}
E_k (u) & = \frac{1}{(\sqrt{k})^k}\int_L |du|^k \vol_L = \int_L \lambda^k \vol_L \geq \int_L u^* \alpha = [\alpha] \cdot u_{*}[L],
\end{align*}
where we have used the fact that $\alpha$ is closed. Thus the $k$-energy of $u$ is bounded from below by a topological quantity, as it depends only on the cohomology class $[\alpha]$ and the homotopy class of $u$. Moreover, by Theorem~\ref{thm:immersion-ineq}, equality holds if and only if $u$ is a Smith immersion. This shows that such maps are local minimizers of $E_k$ and thus are $k$-harmonic.
\end{proof}

We note that Theorem~\ref{thm:Smith-immersion-energy-ineq} still holds if $L$ is noncompact. See Theorem~\ref{thm:Smith-immersion-direct-harmonic}.

\begin{rmk} \label{rmk:immersion-conf-inv}
Since a Smith immersion $u \colon (L^k, g) \to (M^n, h)$ with respect to $\alpha \in \Omega^k (M)$ is in particular a $k$-harmonic map (when $d \alpha = 0$), by the discussion at the end of Section~\ref{sec:p-harmonic}, we expect that the notion of a Smith immersion should depend only on the conformal class $[g]$ of the metric on the domain $L$. Indeed, this is true even without the assumption that $d\alpha = 0$. To see this, suppose $\tilde g = f^2 g$ for some smooth positive function on $L$. From $\eqref{eq:du-norm}$ we get
$$ \wt \lambda^2 = \frac{1}{k} |du|_{\wt g, h}^2 = f^{-2} \frac{1}{k} |du|_{g, h}^2 = f^{-2} \lambda^2, $$
and clearly $\wt \vol_L = f^k \vol_L$. It follows that the Smith immersion equations $u^* \alpha = \lambda^k \vol_L$ and $u^* h = \lambda^2 g$ are \emph{invariant under conformal scaling} of the domain metric $g$ on $L$.
\end{rmk}

\subsection{Direct proof that Smith immersions are $k$-harmonic} \label{sec:Smith-imm-direct}

In Theorem~\ref{thm:Smith-immersion-direct-harmonic} below we show directly that a Smith immersion satisfies the $k$-harmonic map equation, in the sense that $\tau_k (u) = 0$, without assuming $L$ is compact. This argument appeared earlier in~\cite[Section 3.5]{CKM1} under the assumption that $\alpha$ induces a vector cross product $P_{\alpha}$ by raising an index. We provide a slightly modified argument here to show that this assumption was in fact unnecessary. First we need some preliminary results.

\begin{prop} \label{prop:Smith-immersion-direct-prelim}
Let $u \colon (L^k, g) \to (M^n, h)$ be a Smith immersion with respect to the calibration form $\alpha \in \Omega^{k}$ on $M$. Then we have
\begin{equation} \label{eq:Smith-imm-P-version}
P_{\alpha} \circ \Lambda^{k-1}(du) \circ \star_L = \frac{ (-1)^{k-1}}{(\sqrt{k})^{k-2}}|du|^{k-2} du.
\end{equation}
\end{prop}
\begin{proof}
The equation is trivially satisfied at points where $du$ is zero. Let $x \in L^0$. Also, recall that we necessarily have $u^* h = \lambda^2 g$. Let $e_1, \ldots, e_k$ be an oriented orthonormal basis for $T_x L$. We compute
\begin{align*}
h ( P_{\alpha} (\Lambda^{k-1} du) (e_1 \wedge \cdots \wedge e_{k-1}), du (e_k) ) & = \alpha ( (\Lambda^{k-1} du) (e_1 \wedge \cdots \wedge e_{k-1}), du (e_k) ) \\ \nonumber
& = u^* \alpha (e_1 \wedge \cdots \wedge e_k) \\ \nonumber
& = \lambda^k \vol_L (e_1 \wedge \cdots \wedge e_k) \\ \nonumber
& = \lambda^k g (\star (e_1 \wedge \cdots \wedge e_{k-1}), e_k) \\ \nonumber
& = \lambda^{k-2} u^* h (\star (e_1 \wedge \cdots \wedge e_{k-1}), e_k) \\ \label{eq:Smith-imm-direct-temp}
& = \lambda^{k-2} h (du (\star (e_1 \wedge \cdots \wedge e_{k-1})), du (e_k) ).
\end{align*}
Denoting $A := P_{\alpha} (\Lambda^{k-1} du) \colon \Lambda^k (T_x L) \to T_{u(x)} M$, the above says
\begin{equation} \label{eq:Smith-imm-direct-temp}
h ( A (e_1 \wedge \cdots \wedge e_{k-1}), du (e_k) ) = \lambda^{k-2} h (du (\star (e_1 \wedge \cdots \wedge e_{k-1})), du (e_k) ).
\end{equation}
Recall that $du (T_x L)$ is $\alpha$-calibrated by Theorem~\ref{thm:immersion-ineq}. Suppose $w \in (\im du_x)^{\perp}$. Then we have
\begin{align*}
h ( A (e_1 \wedge \cdots \wedge e_{k-1}), w) & = h( P_{\alpha} (du (e_1) \wedge \cdots \wedge du (e_{k-1}), w ) \\
& = \alpha( du (e_1), \ldots, du (e_{k-1}), w) = 0
\end{align*}
by Lemma~\ref{lemma:cousin}. Hence we have shown that $\im A \subseteq \im du_x$. It therefore follows from~\eqref{eq:Smith-imm-direct-temp} and the fact that $du_x$ is injective that
$$ P_{\alpha} \circ (\Lambda^{k-1} du) = \lambda^{k-2} du \circ \star_L. $$
Using that $\star^2 = (-1)^{k-1}$ on $1$-forms, we obtain the desired result.
\end{proof}

In the case where $P_{\alpha}$ is a \emph{vector cross product}, it was shown in~\cite[Proposition 2.32]{CKM1} that~\eqref{eq:Smith-imm-P-version} is \emph{equivalent} to our Smith immersion equation~\eqref{eq:Smith-immersion-eq}. In fact this holds in general.

\begin{prop}
We have shown that if $u \colon (L^k,g) \to (M^n, h,\alpha)$ is a Smith immersion, then
\begin{equation} \label{eq:Smith-imm-second-eqq}
P_{\alpha} \circ \Lambda^{k-1}(du) \circ \star_L =(-1)^{k-1} \frac{|du|^{k-2}}{\sqrt{k}^{k-2}} du.
\end{equation}
The converse also holds. That is, if~\eqref{eq:Smith-imm-second-eqq} holds, then $u$ is a Smith immersion.
\end{prop}
\begin{proof}
Let $x \in L.$ If $du_x = 0$, which satisfies~\eqref{eq:Smith-imm-second-eqq} at $x$, then $u$ is a Smith immersion at $x$. Now assume $du_x \neq 0$. Let  $e_1, \dots, e_k$ be a oriented orthonormal basis of $T_x L$. Let $i, j \in \{ 1, \ldots, k \}$. Then we have
$$ \star_L e_i = (-1)^{i-1} e_1 \wedge \dots \wedge \widehat{e_i} \wedge \dots \wedge e_k.$$
Evaluating both sides of~\eqref{eq:Smith-imm-second-eqq} on $e_i$ and taking inner product with $du(e_j)$ we get
\begin{align*}
(-1)^{k-1} \lambda^{k-2} h( du(e_i), du(e_j) ) & = (-1)^{i-1} h( P_{\alpha} (du(e_1) \wedge \dots \wedge \widehat{du(e_i}) \wedge \dots \wedge du(e_k) ), du(e_j) ) \\
& = (-1)^{i-1} u^{*} \alpha(e_1 \wedge \dots \wedge \widehat{e_i} \wedge \dots \wedge e_k \wedge e_j) \\
& = (-1)^{k-1} u^{*} \alpha(e_1 \wedge \dots \wedge e_j \wedge \dots \wedge e_k).
\end{align*}
We deduce that
$$ \begin{cases} \quad \lambda^{k-2} h(du(e_i), du(e_j)) \vol_L = u^{*} \alpha & \text{if $i=j$}, \\ \quad h( du(e_i), du(e_j) ) = 0 & \text{if $i \neq j$}. \end{cases} $$
Using the above we compute
\begin{align*}
u^{*} \alpha & = \frac{1}{k} \lambda^{k-2} \sum_{i} h( du(e_i), du(e_i) ) \vol_L \\
& = \frac{1}{k} \lambda^{k-2} \sum_{i,j} h( du(e_i), du(e_j) ) \vol_L \\
& = \frac{1}{k} \lambda^{k-2} |du|^2 \vol_L = \lambda^{k} \vol_L,
\end{align*}
and thus $u$ is a Smith immersion in the sense of Definition~\ref{defn:Smith-immersion}.
\end{proof}

\begin{lemma}\label{lemma:Smith-immersion-direct-pre2}
Let $u \colon (L^k, g) \to (M^n, h)$ be a Smith immersion with respect to the calibration $\alpha$. Then $u^* (\nabla_V \alpha) = 0$ for any $V \in \Gamma(TM)$.
\end{lemma}
\begin{proof}
The equation is trivially satisfied at points where $du$ is zero. Let $x \in L^0$. Let $e_1, \ldots, e_k$ be an oriented orthonormal basis for $T_x L$. Then from the proof of Theorem~\ref{thm:immersion-ineq}, we have that $\frac{1}{\lambda} du (e_1), \ldots, \frac{1}{\lambda} du (e_k)$ is an oriented orthonormal basis for $du(T_x L) \subseteq T_{u(x)} M$, which is calibrated by $\alpha$. Thus we have
\begin{align} \nonumber
\tfrac{1}{\lambda^k } u^* (\nabla_V \alpha) ( e_1\wedge \cdots \wedge e_k) & = u^* (\nabla_V \alpha) ( \tfrac{1}{\lambda} e_1 \wedge \cdots \wedge \tfrac{1}{\lambda} e_k) \\ \nonumber
& = V \big( \alpha (\tfrac{1}{\lambda} du (e_1)\wedge \cdots \wedge \tfrac{1}{\lambda} du (e_k)) \big) \\ \label{eq:Smith-immersion-direct-pre2-temp}
& \qquad {} - \sum_{j=1} \alpha (\tfrac{1}{\lambda} du (e_1) \wedge \cdots \wedge \nabla_V (\tfrac{1}{\lambda} du (e_j)) \wedge \cdots \wedge \tfrac{1}{\lambda} du (e_k) ).
\end{align}
The first term in~\eqref{eq:Smith-immersion-direct-pre2-temp} vanishes because $\alpha$ calibrates $du(T_x L)$. By skew-symmetry of $\alpha$, the only component of $\nabla_V (\frac{1}{\lambda} du (e_j))$ in the span of $\frac{1}{\lambda} du (e_1), \ldots, \frac{1}{\lambda} du (e_k)$ which can contribute to
$$ \alpha (\tfrac{1}{\lambda} du (e_1) \wedge \cdots \wedge \nabla_V (\tfrac{1}{\lambda} du (e_j)) \wedge \cdots \wedge \tfrac{1}{\lambda} du (e_k) ) $$
is the $\frac{1}{\lambda} du (e_j)$ component. But since $\frac{1}{\lambda} du (e_j)$ has constant (unit) length, the covariant derivative $\nabla_V (\frac{1}{\lambda} du (e_j))$ is orthogonal to $\frac{1}{\lambda} du (e_j)$. We deduce that
$$ \alpha (\tfrac{1}{\lambda} du (e_1) \wedge \cdots \wedge \nabla_V (\tfrac{1}{\lambda} du (e_j)) \wedge \cdots \wedge \tfrac{1}{\lambda} du (e_k) ) = \alpha (\tfrac{1}{\lambda} du (e_1) \wedge \cdots \wedge w \wedge \cdots \wedge \tfrac{1}{\lambda} du (e_k) ) $$
for some vector $w$ orthogonal to the $\alpha$-calibrated $k$-plane spanned by $\frac{1}{\lambda} du (e_1), \ldots, \frac{1}{\lambda} du (e_k)$. It then follows from Lemma~\ref{lemma:cousin} that each of the terms in the last line of~\eqref{eq:Smith-immersion-direct-pre2-temp} also vanish, so $u^* (\nabla_V \alpha) = 0$.
\end{proof}

The next result is exactly~\cite[Proposition 3.20]{CKM1}, but with a harmless sign error corrected. We include it for completeness and comparison with Theorem~\ref{thm:Smith-submersion-direct-harmonic} in the case of Smith submersions.

\begin{thm} \label{thm:Smith-immersion-direct-harmonic}
Let $u \colon (L^k, g) \to (M^n, h)$ be a Smith immersion with respect to the calibration form $\alpha \in \Omega^{k}$. If $d \alpha = 0$, then $u$ is $k$-harmonic in the sense that $\tau_k (u) = 0$.
\end{thm}
\begin{proof}
We show that the $k$-tension $\tau_k (u)$ of equation~\eqref{eq:p-tension} vanishes at any point $x \in L$. Let
$$ B = P_{\alpha} \circ \Lambda^{k-1} (du) \circ \star_L \in \Gamma (T^* L \otimes u^* TM). $$
By Proposition~\ref{prop:Smith-immersion-direct-prelim}, it suffices to show that $\Div(B) = 0$, which is a smooth section of $u^* TM$. Let $\mu$ denote the Riemannian volume form on $L$, and identify $1$-forms and vector fields using the musical isomorphism. Recall that $\star v = v \hk \mu$ for any vector field $v$ on $L$, so $(\star v)_{i_1 \cdots i_{k-1}} = v_j \mu_{j i_1 \cdots i_{k-1}}$. We also have $(P_{\alpha})_{b_1 \cdots b_{k-1} a} = \alpha_{b_1 \cdots b_{k-1} a}$.

Take Riemannian normal coordinates $\frac{\partial}{\partial x^i}, \frac{\partial}{\partial y^a}$ centred at $x$ and $u(x)$ respectively. At the point $x$, we compute
\begin{align*}
\Div(B)_{a} & = (\nabla_j B)_{a j} \\
& = \nabla_j (P_{\alpha} \circ \Lambda^{k-1} (du) \circ \star_L)_{aj} \\
& = \frac{1}{(k-1)!} \nabla_j (P_{\alpha} \circ \Lambda^{k-1}(du))_{i_1 \cdots i_{k-1} a} \mu_{j i_1 \cdots i_{k-1}} \\
& = \frac{1}{(k-1)!} \nabla_j \left( \frac{\partial u^{b_1}}{\partial x^{i_1}} \cdots \frac{\partial u^{b_{k-1}}}{\partial x^{i_{k-1}}} (P_{\alpha})_{b_1 \cdots b_{k-1} a} \right) \mu_{j i_1 \cdots i_{k-1} } \\
& = \frac{1}{(k-1)!} \nabla_j \left( \frac{\partial u^{b_1}}{\partial x^{i_1}} \cdots \frac{\partial u^{b_{k-1}}}{\partial x^{i_{k-1}}} \alpha_{b_1 \cdots b_{k-1} a} \right) \mu_{j i_1 \cdots i_{k-1}} \\
& = \frac{1}{(k-1)!} \frac{\partial u^{b_1}}{\partial x^{i_1}} \cdots \frac{\partial u^{b_{k-1}}}{\partial x^{i_{k-1}}} (\nabla_j \alpha_{b_1 \cdots b_{k-1} a}) \mu_{j i_1 \cdots i_{k-1}} \\
& \qquad {} + \frac{1}{(k-1)!} \sum_{\ell=1}^{k-1} \frac{\partial^2 u^{b_{\ell}}}{\partial x^j \partial x^{i_{\ell}}} \frac{\partial u^{b_1}}{\partial x^{i_1}} \cdots \widehat{\frac{\partial u^{b_{\ell}}}{\partial x^{i_{\ell}}}} \cdots \frac{\partial u^{b_{k-1}}}{\partial x^{i_{k-1}}} \alpha_{b_1 \cdots b_{k-1} a} \mu_{j i_1 \cdots i_{k-1}},
\end{align*}
where the $\widehat{\phantom{a}}$ as usual denotes omission. The second term vanishes by (skew)-symmetry in $j, i_{\ell}$. For the first term, we have
$$ \nabla_j \alpha = \nabla_{\frac{\partial}{\partial x^{j}}} \alpha = \frac{\partial u^{b_k}}{\partial x^j} \nabla_{\frac{\partial}{\partial y^{b_k}}} \alpha, $$
which we write as $\frac{\partial u^{b_k}}{\partial x^j} \nabla_{b_k} \alpha$. Thus we have
$$ \Div(B)_{a} = \frac{1}{(k-1)!}\frac{\partial u^{b_1}}{\partial x^{i_1}} \cdots \frac{\partial u^{b_{k-1}}}{\partial x^{i_{k-1}}}\frac{\partial u^{b_k}}{\partial x^j} (\nabla_{b_k} \alpha_{b_1 \cdots b_{k-1} a}) \mu_{j i_1 \cdots i_{k-1}}. $$
Relabelling $j$ as $i_k$, we have
$$ \Div(B)_{a} = \frac{(-1)^{k-1}}{(k-1)!} \frac{\partial u^{b_1}}{\partial x^{i_1}} \cdots \frac{\partial u^{b_k}}{\partial x^{i_k}} (\nabla_{b_k} \alpha_{b_1 \cdots b_{k-1} a}) \mu_{i_1 \cdots i_k}. $$
By the skew-symmetry of $\mu$, if we swap $b_{\ell}$ and $b_m$ in the factor $(\nabla_{b_k} \alpha_{b_1 \cdots b_{k-1} a})$ above, the sign of the right hand side changes. We therefore can write
$$ \Div(B)_{a} = \frac{(-1)^{k-1}}{(k-1)!} \frac{\partial u^{b_1}}{\partial x^{i_1}} \cdots \frac{\partial u^{b_k}}{\partial x^{i_k}} \frac{1}{k} \sum_{\ell = 1}^k (\nabla_{b_{\ell}} \alpha_{b_1 \cdots b_{\ell-1} a b_{\ell+1} \cdots b_{k}}) \mu_{i_1 \cdots i_k} $$
because for each $\ell$ when we swap $a$ with $b_k$ and then $b_k$ with $b_{\ell}$ we introduce two minus signs which cancel. Finally we use the fact that $\alpha$ is closed to write
$$ 0 = (d\alpha)_{ab_1 \cdots b_k} = \nabla_a \alpha_{b_1 \cdots b_k} - \sum_{\ell = 1}^k (\nabla_{b_{\ell}} \alpha_{b_1 \cdots b_{\ell-1} a b_{\ell+1} \cdots b_{k}}). $$
Combining these we obtain
\begin{align*}
\Div(B)_{a} & = \frac{(-1)^{k-1}}{k!} \frac{\partial u^{b_1}}{\partial x^{i_1}} \cdots \frac{\partial u^{b_k}}{\partial x^{i_k}} \nabla_a \alpha_{b_1 \cdots b_k} \mu_{i_1 \cdots i_k }\\
& = (u^* \nabla_a \alpha)_{i_1 \cdots i_k} \mu_{i_1 \cdots i_k},
\end{align*}
which vanishes by Lemma~\ref{lemma:Smith-immersion-direct-pre2}, completing the proof.
\end{proof}

\section{Smith submersions} \label{sec:Smith-submersions}

We introduce a new class of maps $u \colon (M^n, h) \to (L^k, g)$ between Riemannian manifolds with $n \geq k$, where the \emph{domain} is equipped with a calibration form $\alpha$ of degree $n-k$. These maps are a special class of $k$-harmonic maps satisfying a first order nonlinear differential equation, and have the property that when $d\alpha = 0$, the \emph{smooth fibres} are $\alpha$-calibrated submanifolds of $M$.

In this section, $u \colon (M^n, h) \to (L^k, g)$ is a \emph{surjective} smooth map between Riemannian manifolds, with $n \geq k$. Recall that $u \colon (M^n, h) \to (L^k, g)$ is a \emph{submersion} if $\rank(du_x) = k$ for all $x \in M$.

\subsection{(Weakly) conformally horizontal submersions} \label{sec:weakly-conformally-horizontal}

In order to be able to define the submersion analogue of ``weakly conformal'', we need to first recall the horizontal/vertical splitting of $TM$ associated to a submersion $u \colon M \to L$.

\begin{defn} \label{defn:horizontal/vertical}
Let $u \colon (M^n, h) \to (L^k, g)$ be a smooth surjection. Let $M^{0} \subseteq M$ be the open set where $|du| \neq 0$. Suppose that the restriction $u|_{M^0} \colon M^0 \to L$ is a submersion, so that $\rank(du_x) = k$ for all $x \in M^0$. Then the tangent bundle $TM^{0}$ of $M^0$ decomposes as
$$ TM^{0} = (\ker du) \oplus_{\perp} (\ker du)^{\perp}, $$
where $\ker du = VM^0$ is the \emph{vertical} subbundle, which has rank $n-k$, and $(\ker du)^{\perp} = HM^0$ is the \emph{horizontal} subbundle, which has rank $k$.

It follows that an $m$-tensor $\alpha \in \cT^m$ on $M^0$ is a smooth section of
$$ \bigoplus_{p+q = m} ( \ker du )^{\otimes p} \otimes (( \ker du)^{\perp})^{\otimes q}, $$
with $p \leq n-k, q \leq k$. We denote by $\alpha^{(p,q)}$ the component of $\alpha$ which lies in
$$ \cT^{(p,q)} \coloneqq \Gamma( (\ker du)^{\otimes p} \otimes ((\ker du)^{\perp})^{\otimes q} ) $$
and we say that $\alpha^{(p,q)}$ is of \emph{type} $(p,q)$.

It follows that the metric $h$ on $M^0$ decomposes as $h = h^{2,0} + h^{0,2}$, where $h^{2,0}$ is the metric on the vertical subbundle $\ker du$, and $h^{0,2}$ is the metric on the horizontal subbundle $(\ker du)^{\perp}$. In particular, we have
\begin{equation} \label{eq:tr-h02}
\tr_h (h^{0,2}) = k.
\end{equation}
Finally, we use $\Omega^{(p,q)}$ to denote the totally skew-symmetric elements of $\cT^{(p,q)}$.
\end{defn}

\begin{defn} \label{defn:weakly-horizontally-conformal-submersion}
A smooth surjection $u \colon (M^n, h) \to (L^k, g)$ is called \emph{(weakly) horizontally conformal} if for every point $x \in M$, we either have $du_x = 0$, or if $du_x \neq 0$, then $\rank (du_x) = k$ is maximal and
$$ u^* g = \lambda^2 h^{(0,2)} $$
for some smooth function $\lambda > 0$ on $M^0$. We can extend $\lambda^2$ by zero to obtain a continuous non-negative function on $M$. This function $\lambda$ is called the dilation. It then follows from~\eqref{eq:du-norm} that necessarily $\lambda^2 = \frac{1}{k} |du|^2$.

When $M^0 = M$, we say that $u$ is a \emph{horizontally conformal submersion}. A submersion $u \colon (M^n, h) \to (L^k, g)$ is called a \emph{Riemannian submersion} if $u^* g = h^{(0,2)}$ on $M$, or equivalently if it is a horizontally conformal submersion with dilation $\lambda = 1$.
\end{defn}

\begin{rmk} \label{rmk:smooth-fibres}
Let $u \colon (M^n, h) \to (L^k, g)$ be weakly horizontally conformal. Restricted to the open subset $M^0$, the map $u|_{M^0}$ is a submersion, and thus by the implicit function theorem each fibre $M^0 \cap u^{-1} \{ u(x) \}$ for $x \in M^0$ is a smooth $(n-k)$-dimensional submanifold of $M^0$.
\end{rmk}

\begin{rmk} \label{rmk:isometry}
Let $u \colon (M^n, h) \to (L^k, g)$ be a smooth surjection. Over $M^0$ we get a canonical orientation on the horizontal subbundle $(\ker du)^{\perp}$ from the class $[u^* \vol_L]$. Then the vertical subbundle $\ker du$ inherits a unique orientation such that $\vol_{\ker du} \wedge \vol_{(\ker du)^{\perp}} = \vol_M$.

If $u$ is (weakly) horizontally conformal, then by Definition~\ref{defn:weakly-horizontally-conformal-submersion}, we have that for any $x \in M^0$, the map
$$ (du)_x \colon \big( ( \ker du_x)^{\perp}, \lambda^2(x) h^{(0,2)}_x \big) \cong \big( T_{u(x)} L, g_{u(x)} \big)$$
is an orientation preserving isometry.
\end{rmk}

For the remainder of this section, we assume that $u \colon (M^n, h) \to (L^k, g)$ is horizontally conformal. (Equivalently, it is weakly conformally horizontal and we work only on the open subset $M^0$ where it is horizontally conformal.) We collect several results that are needed to study Smith submersions.

\begin{lemma} \label{lemma:pullbacks-are-horizontal}
Let $\beta \in \Omega^p (L)$. Then $u^*\beta$ is of type $(0,p)$.
\end{lemma}
\begin{proof}
Let $v_1, \ldots, v_p \in \Gamma(TM)$. Then $(u^*\beta) (v_1, \ldots, v_p) = \beta (du(v_1), \ldots, du(v_p))$, so if at least one of the $v_i$ lies in $ \ker du$, then $(u^*\beta) (v_1, \ldots, v_p) = 0$.
\end{proof}

\begin{lemma} \label{lemma:star-pq}
Let $\alpha \in \Omega^{(p,q)} (M)$. Then $\star \alpha \in \Omega^{(n-k-p, k-q)} (M)$. Moreover, for any form $\beta$, we have $(\star \beta)^{(n-k-p, k-q)} = \star(\beta^{(p,q)})$.
\end{lemma}
\begin{proof}
This follows from the fact that $\vol_M \in \Omega^{(n-k, k)} (M)$.
\end{proof}

\begin{lemma} \label{lemma:hook-type}
Let $\alpha \in \Omega^{(p,q)} (M)$. Then for any $v \in \Gamma(TM)$, the form $v^{(1,0)} \hk \alpha$ is of type $(p-1, q)$ and the form $v^{(0,1)} \hk \alpha$ is of type $(p, q-1)$.
\end{lemma}
\begin{proof}
This is clear from definition of the interior product.
\end{proof}

\begin{lemma} \label{lemma:PP-tr}
Let $\alpha \in \Omega^{(0,k)} (M)$. Let $P_{\alpha}$ be as in Definition~\ref{defn:P-from-alpha}, and let $P_{\alpha}^{\top}$ be its adjoint map as in Proposition~\ref{prop:P-adjoint}. Then we have
$$ P_{\alpha} P_{\alpha}^{\top} = |\alpha|^2 \pi^{(0,1)}, $$
where $\pi^{(0,1)} \colon \Gamma(TM) \to \Gamma(TM^{(0,1)})$ is the orthogonal projection.
\end{lemma}
\begin{proof}
First, note that since $\alpha$ is of type $(0,k)$, and the metric $h$ on $TM$ is of type $(2,0) + (0,2)$, the map $P_{\alpha}$ takes values in the horizontal subbundle $TM^{(0,1)} = (\ker du)^{\perp}$. Consider any $v \in \Gamma(TM)$ and $w \in \Gamma(TM^{(0,1)})$. By Proposition~\ref{prop:P-adjoint} we have $P_{\alpha}^{\top} v = (-1)^{k-1} v \hk \alpha$. Hence we have
\begin{align*}
g (P_{\alpha} P_{\alpha}^{\top} v, w) & = (-1)^{k-1}  g(P (v \hk \alpha), w) \\
& = (-1)^{k-1} \alpha( (v \hk \alpha) \wedge w) \\
& = g (\alpha, w\wedge (v \hk \alpha)).
\end{align*}
Recall that $v \hk (w \wedge \alpha) = (v \hk w) \alpha - w \wedge (v \hk \alpha)$, and thus $w \wedge (v \hk \alpha) = g(v, w) \alpha$ because $w \wedge \alpha = 0$ since it is of type $(0,k+1)$. Hence, we get
$$ g (P_{\alpha} P_{\alpha}^{\top} v, w) = g(v, w) |\alpha|^2, $$
and the result follows.
\end{proof}

\subsection{Smith submersions and the energy inequality} \label{sec:Smith-sub-energy}

We can now consider the notion of a Smith submersion.

\begin{thm} \label{thm:submersion-ineq}
Let $u \colon (M^n, h) \to (L^k, g)$ be a smooth surjection. Let $\alpha \in \Omega^{n-k} (M)$ be a calibration. Then
\begin{equation} \label{eq:submersion-ineq}
\alpha \wedge u^* \vol_L \leq \lambda^k \vol_M, \quad \text{where $\lambda = \frac{|du|}{\sqrt{k}}$}.
\end{equation}
Moreover, equality holds if and only if:
\begin{itemize} \setlength\itemsep{-1mm}
\item $u^* g = \lambda^2 h^{(0,2)}$ (so $u$ is a weakly horizontally conformal submersion) and,
\item the fibres of the restriction of $u$ to $M^0$ are calibrated with respect to $\alpha$.
\end{itemize}
\end{thm}
\begin{proof}
We trivially have equality at points where $du$ is zero. Let $x \in M^0$. If $du_x$ is not maximal rank, then $u^* \vol_L$ vanishes, while $\lambda > 0$, so the inequality~\eqref{eq:submersion-ineq} is satisfied and indeed is always a \emph{strict} inequality at such points.

Now consider $x \in M^0$ such that $du_x$ has maximal rank $k$. Let $e_1, \ldots, e_k$ be an oriented orthonormal basis of $ (\ker du_x)^{\perp}$ and $\tilde{e}_1, \ldots, \tilde{e}_{n-k}$ be an oriented orthonormal basis of $\ker du_x$. With our choice of orientations from Remark~\ref{rmk:isometry} we have $\vol_M = \tilde{e}_1 \wedge \cdots \wedge \tilde{e}_{n-k} \wedge e_1 \wedge \cdots \wedge e_k$. Then we have
\begin{align*}
& (\alpha \wedge u^* \vol_L) (\tilde{e}_1 \wedge \cdots \wedge \tilde{e}_{n-k} \wedge e_1 \wedge \cdots \wedge e_k) & & \\
& \qquad {} = \alpha (\tilde{e}_1 \wedge \cdots \wedge \tilde{e}_{n-k}) u^* \vol_L (e_1 \wedge \cdots \wedge e_k) & & \text{(by Lemma~\ref{lemma:pullbacks-are-horizontal})} \\
& \qquad {} \leq 1 \cdot \vol_L ((\Lambda^k du) (e_1 \wedge \cdots \wedge e_k)) & & \text{(because $\alpha$ is a calibration)} \\
& \qquad {} = |(\Lambda^k du) (e_1 \wedge \cdots \wedge e_k)| & & \\
& \qquad {} = |\Lambda^k du| \, | (e_1 \wedge \cdots \wedge e_k)| & & \\
& \qquad {} \leq \lambda^k & & \text{(by Proposition~\ref{prop:Hadamard})},
\end{align*}
which concludes the proof of~\eqref{eq:submersion-ineq}. 

Equality holds if and only if equality holds in the two inequalities of the above computation. If the second inequality above is an equality, then by Proposition~\ref{prop:Hadamard} we have $u^* g = \lambda^2 h^{(0,2)}$, so $u$ is weakly horizontally conformal. Let $x \in M^0$ and let $\tilde e_1, \ldots, \tilde e_{n-k}$ be an orthonormal frame for $\ker du_x$. If the first inequality above is an equality, then we see that we must have $\alpha (\tilde e_1 \wedge \cdots \wedge \tilde e_{n-k}) = 1$. That is, the smooth fibre $M^0 \cap u^{-1} \{u(x)\}$ is calibrated with respect to $\alpha$.
\end{proof}

\begin{defn} \label{defn:Smith-submersion}
If equality holds in~\eqref{eq:submersion-ineq}, we say that $u$ is a \textbf{Smith submersion} with respect to $\alpha$. That is, a Smith submersion with respect to $\alpha$ is a smooth map $u \colon (M^n, h) \to (L^k, g)$ such that
\begin{equation} \label{eq:Smith-submersion-eq}
\alpha \wedge u^* \vol_L = \frac{1}{(\sqrt{k})^k} |du|^k \vol_M, \qquad u^* g = \frac{1}{k} |du|^2 h^{(0,2)},
\end{equation}
at all points on $M$. [However, recall that the first equation automatically implies the second equation.] Note that, strictly speaking, a Smith submersion is only actually a submersion on the open subset $M^0 = \{ x \in M : du_x \neq 0 \}$ of $M$.
\end{defn}

Before we prove the Smith submersion energy inequality in Theorem~\ref{thm:Smith-submersion-energy-ineq} below, which is analogous to Theorem~\ref{thm:Smith-immersion-energy-ineq} for Smith immersions, we first show that in the Smith submersion case we can rewrite the equation in a useful alternative form.

\begin{lemma} \label{lemma:submersion-star-alpha}
Let $u \colon (M^n, h) \to (L^k, g)$ be weakly horizontally conformal with dilation $\lambda$. Let $\alpha \in \Omega^{n-k} (M)$ be a calibration, so $\star \alpha \in \Omega^k (M)$ is also a calibration by Proposition~\ref{prop:star-calibration}. At at point $x$ where $du_x \neq 0$, the following are equivalent:
\begin{enumerate}[(i)]
\item $u^* \vol_L = \lambda^k (\star \alpha)^{(0,k)}$,
\item $(\ker du)^{\perp}$ is calibrated with respect to $\star \alpha$,
\item $\ker du$ is calibrated with respect to $\alpha$.
\end{enumerate}
\end{lemma}
\begin{proof}
\emph{(i)}$\iff$\emph{(ii)}. Let $e_1, \ldots, e_k$ be an oriented orthonormal basis of $(\ker du_x)^{\perp}$. Then since
$$ (\star \alpha) (e_1, \ldots, e_k) = (\star \alpha)^{(0,k)} (e_1, \ldots, e_k) \quad \text{and} \quad (u^* \vol_L) (e_1, \ldots, e_k) = \lambda^k, $$
we have that $u^* \vol_L = \lambda^k (\star \alpha)^{(0,k)}$ if and only if $(\star \alpha) (e_1, \ldots, e_k) = 1$ if and only if $(\ker du)^{\perp}$ is calibrated with respect to $\star \alpha$.

\emph{(ii)}$\iff$\emph{(iii)}. Let $\tilde{e}_1, \ldots, \tilde{e}_{n-k}$ be an oriented orthonormal basis of $\ker du_x$. Note that
$$ \vol_M = \tilde{e}_1 \wedge \cdots \wedge \tilde{e}_{n-k} \wedge e_1 \wedge \cdots \wedge e_k. $$
Thus we have
\begin{align*}
\alpha (\tilde{e}_1, \ldots, \tilde{e}_{n-k}) & = h (\alpha, \tilde{e}_1 \wedge \cdots \wedge \tilde{e}_{n-k}) \\
& = h (\star \alpha, \star (\tilde{e}_1 \wedge \cdots \wedge \tilde{e}_{n-k})) = h (\star \alpha, e_1 \wedge \cdots \wedge e_k) = (\star \alpha) (e_1, \ldots, e_k),
\end{align*}
and the result follows.
\end{proof}

\begin{cor} \label{cor:Smith-submersion-alternative-form}
Let $u \colon (M^n, h) \to (L^k, g)$ be a smooth surjection. Let $\lambda = \frac{|du|}{\sqrt{k}}$ and $\alpha \in \Omega^{n-k} (M)$ be a calibration. Then the following are equivalent:
\begin{enumerate}[(i)]
\item \, $u^* \vol_L = \lambda^k (\star \alpha)^{(0,k)}$ \, and \, $u^* g = \lambda^2 h^{(0,2)}$,
\item \, $\alpha \wedge u^* \vol_L = \lambda^k \vol_M$.
\end{enumerate}
\end{cor}
\begin{proof} Both equations are trivially satisfied at the points where $du$ is zero. Let $x \in M^0$.

Suppose that \emph{(i)} holds. By Lemma~\ref{lemma:submersion-star-alpha}, we have that $(\ker du)^{\perp}$ is calibrated with respect to $\star \alpha$. Combining this with $u^* g = \lambda^2 h^{(0,2)}$ and using Theorem~\ref{thm:submersion-ineq}, we obtain $\alpha \wedge u^* \vol_L = \lambda^k \vol_M$.

Conversely, suppose \emph{(ii)} holds. From Theorem~\ref{thm:submersion-ineq} we know that $u$ is horizontally conformal and $\alpha$ calibrates $ \ker du$.  Hence by Lemma~\ref{lemma:submersion-star-alpha} we also have $u^* \vol_L = \lambda^k (\star \alpha)^{(0,k)}$, so \emph{(i)} holds.
\end{proof}

\begin{rmk} \label{rmk:Smith-submersion-two-forms}
Corollary~\ref{cor:Smith-submersion-alternative-form} establishes two equivalent formulations of Smith submersion. The original definition of Smith submersion in~\eqref{eq:Smith-submersion-eq} is precisely \emph{(ii)} of Corollary~\ref{cor:Smith-submersion-alternative-form}, since the first equation in~\eqref{eq:Smith-submersion-eq} implies the second. However, in the alternative formulation \emph{(i)} of Corollary~\ref{cor:Smith-submersion-alternative-form}, we need \emph{both} equations. The first does not in general imply the second.

Moreover, the original definition in~\eqref{eq:Smith-submersion-eq} arises as the case of equality in the general inequality of~\eqref{eq:submersion-ineq}. Similarly, we can show that \emph{if we assume the second equation in (i) of Corollary~\ref{cor:Smith-submersion-alternative-form}}, then we claim that we always have the inequality
\begin{equation} \label{eq:Smith-alt-ineq}
u^* \vol_L \geq \lambda^k (\star \alpha)^{(0,k)}.
\end{equation}
However, the inequality~\eqref{eq:Smith-alt-ineq} need not hold in general if we do not assume $u^* g = \lambda^2 h^{(0,2)}$.

To see that~\eqref{eq:Smith-alt-ineq} holds if $u^* g = \lambda^2 h^{(0,2)}$, note that both sides are sections of the oriented line bundle $\Lambda^k (\ker(du)^{\perp})^*$ whose space of sections is $\Omega^{(0,k)}$. Hence we can compare any two elements. Clearly the inequality holds on $M \setminus M^{0}$ as both sides are zero. Let $x \in M^0$. Let $e_1, \ldots, e_k$ be an oriented orthonormal basis of $(\ker du_x)^{\perp}$. Then since $u^* g = \lambda^2 h^{(0,2)}$ we have
$$ u^* \vol_L (e_1 \wedge \cdots \wedge e_k) = \lambda^k, $$
and since $\star \alpha$ is also a calibration we have
$$ \lambda^k (\star \alpha)^{(0,k)} (e_1\wedge \cdots \wedge e_k) = \lambda^k (\star \alpha) (e_1 \wedge \cdots \wedge e_k) \leq \lambda^k. $$
Thus the inequality~\eqref{eq:Smith-alt-ineq} holds if $u^* g = \lambda^2 h^{(0,2)}$.\\
Finally, as in the immersion case, there is another equivalent form of the Smith equation, which we prove in Propositions~\ref{prop:Smith-submersion-direct-prelim} and~\ref{prop:Smith-submersion-direct-prelim2}.
\end{rmk}

\begin{thm}[Energy Inequality] \label{thm:Smith-submersion-energy-ineq}
Let $\alpha \in \Omega^{n-k} (M)$ be a \emph{closed} calibration. Let $u \colon (M^n, h) \to (L^k, g)$ be a Smith submersion with respect to $\alpha$. Suppose $M$ is compact. Then $u$ is $k$-harmonic in the sense that it is a critical point of $E_k$.
\end{thm}
\begin{proof}
For any smooth map $u \colon (M^n, h) \to (L^k, g)$, let $\lambda = \frac{1}{\sqrt{k}} |du|$. Using~\eqref{eq:submersion-ineq} we have
\begin{align*}
E_k (u) & = \frac{1}{(\sqrt{k})^k}\int_M |du|^k \vol_M = \int_M \lambda^k \vol_M \geq \int_M \alpha \wedge u^* \vol_L = ( [\alpha] \cup u^*[\vol_L] ) \cdot [M],
\end{align*}
where we have used the fact that $\alpha$ is closed. Thus the $k$-energy of $u$ is bounded from below by a topological quantity, as it depends only on the cohomology class $[\alpha]$ and the homotopy class of $u$. Moreover, by Theorem~\ref{thm:submersion-ineq}, equality holds if and only if $u$ is a Smith submersion. This shows that such maps are local minimizers of $E_k$ and thus are $k$-harmonic.
\end{proof}

We note that Theorem~\ref{thm:Smith-submersion-energy-ineq} still holds if $M$ is noncompact. See Theorem~\ref{thm:Smith-submersion-direct-harmonic}.

\begin{rmk} \label{rmk:submersion-conf-inv}
Smith submersions also enjoy a sort of ``conformal invariance'', but it is slightly more complicated. (This is expected, because a Smith submersion $u \colon (M^n, h) \to (L^k, g)$ with respect to $\alpha \in \Omega^{n-k} (M)$ is in particular a $k$-harmonic map (when $d \alpha = 0$), so by the discussion at the end of Section~\ref{sec:p-harmonic}, this notion would depend only on the conformal class $[h]$ of the metric on the domain $M$ only in the particular special case $n=k$.

In general, if $n > k$, we have the following. Let $h = h^{(2,0)} + h^{(0,2)}$ be the decomposition of the metric $h$ on $M$ in terms of the horizontal/vertical splitting as in Definition~\ref{defn:horizontal/vertical}. A \emph{horizontally conformal scaling} of $h$ is a new metric $\wt h = h^{(2,0)} + f^2 h^{(0,2)}$ for some smooth positive function on $L$. (That is, we only conformally scale the \emph{horizontal} part of the metric $h$). Since $du$ is zero on vertical vectors, from $\eqref{eq:du-norm-v2}$ we get
$$ \wt \lambda^2 = \frac{1}{k} |du|_{\wt h, g}^2 = f^{-2} \frac{1}{k} |du|_{h, g}^2 = f^{-2} \lambda^2, $$
and clearly $\wt \vol_M = f^k \vol_M$. It follows that the Smith submersion equations $\alpha \wedge u^* \vol_L = \lambda^k \vol_M$ and $u^* g = \lambda^2 h^{(0,2)}$ are \emph{invariant under horizontally conformal scaling} of the domain metric $h$ on $M$.
\end{rmk}

\subsection{Direct proof that Smith submersions are $k$-harmonic} \label{sec:Smith-sub-direct}

In Theorem~\ref{thm:Smith-immersion-direct-harmonic} below we show directly that a Smith submersion satisfies the $k$-harmonic map equation, in the sense that $\tau_k (u) = 0$, without assuming $M$ is compact. First we need some preliminary results.

\begin{lemma} \label{lemma:Smith-sub-props}
Let $\alpha \in \Omega^{n-k}(M)$ be a calibration. Let $u \colon (M^n, h) \to (L^k, g)$ be a Smith submersion with respect to $\alpha$. Then we have
$$ (\star \alpha)^{(1,k-1)} = 0, \qquad \text{and} \qquad \text{$\nabla \alpha = 0$ on $ \ker du$}. $$
\end{lemma}
\begin{proof}
The first statement follows from Lemma~\ref{lemma:cousin}, because by Corollary~\ref{cor:Smith-submersion-alternative-form} and Lemma~\ref{lemma:submersion-star-alpha}, the form $\star \alpha$ calibrates $(\ker du)^{\perp}$. For the second statement, since $\alpha \in \Omega^{n-k}(M)$ and $\ker du$ is $(n-k)$-dimensional, it is enough to show that
$$ (\nabla_X \alpha) (\tilde{e}_1 \wedge \ldots \wedge \tilde{e}_{n-k}) = 0, $$
for any local orthonormal frame $\tilde{e}_1, \ldots , \tilde{e}_{n-k}$ of $\ker du$. Since by Lemma~\ref{lemma:submersion-star-alpha}, $\alpha$ calibrates $\ker du$, we have that $\alpha (\tilde{e}_1 \wedge \cdots \wedge \tilde{e}_{n-k}) = 1$. Hence we have
\begin{align*}
(\nabla_X \alpha) (\tilde{e}_1 \wedge \cdots \wedge \tilde{e}_{n-k}) & = X \big(\alpha (\tilde{e}_1 \wedge \cdots \wedge \tilde{e}_{n-k}) \big) - \sum_{j=1}^{n-k} \alpha (\tilde{e}_1\wedge \cdots \wedge (\nabla_X \tilde{e}_j) \wedge \cdots \wedge \tilde{e}_{n-k}) \\
& = 0 - \sum_{j=1}^{n-k} \alpha (\tilde{e}_1 \wedge \cdots \wedge (\nabla_X \tilde{e}_j) \wedge \cdots \wedge \tilde{e}_{n-k}).
\end{align*}
Now for any fixed $j$, the term $\alpha (\tilde{e}_1 \wedge \cdots \wedge (\nabla_X \tilde{e}_j)^{(0,1)} \wedge \cdots \wedge \tilde{e}_{n-k})$ vanishes by the first statement. Next note that since the $\tilde{e}_j$ are of norm $1$, the vector field $\nabla_X \tilde{e}_j$ is always orthogonal to $\tilde e_j$, and thus $\tilde{e}_1, \ldots, (\nabla_X \tilde{e}_j)^{(1,0)}, \ldots, \tilde{e}_{n-k}$ are linearly dependent for any $j$, so $\alpha (\tilde{e}_1 \wedge \cdots \wedge (\nabla_X \tilde{e}_j)^{(1,0)} \wedge \cdots \wedge \tilde{e}_{n-k})$ also vanishes, which concludes the proof.
\end{proof}

\begin{prop} \label{prop:Smith-submersion-direct-prelim}
Let $u \colon (M^n, h) \to (L^k, g)$ be a Smith submersion with respect to the calibration form $\alpha \in \Omega^{n-k}$ on $M$. Then we have
\begin{equation} \label{eq:Smith-sub-P-version}
\star_L \Lambda^{k-1}(du) (\cdot \hk \star \alpha) = \frac{(-1)^{k-1}}{(\sqrt{k})^{k-2}}|du|^{k-2} du.
\end{equation}
\end{prop}
\begin{proof}
The equation is trivially satisfied at points where $du$ is zero. Let $x \in M^0$. Also, recall that we necessarily have $u^* g = \lambda^2 h^{(0,2)}$, and that from Corollary~\ref{cor:Smith-submersion-alternative-form} we also have $u^* \vol_L = \lambda^k (\star \alpha)^{(0,k)}$.

For simplicity of notation, let $P_{(0,k)}$ denote $P_{(\star \alpha)^{(0,k)}}$. Note that $P_{(0,k)} \in \Gamma(\Lambda^{(0,k-1)}(TM) \otimes (TM)^{(0,1)}))$. Using this, for any $v_1, \ldots, v_k \in T_x M$ we have
\begin{align*}
g (\star_L (du (v_1) \wedge \cdots \wedge du (v_{k-1}) ), du (v_{k}) ) & = \vol_L (du (v_1) \wedge \cdots \wedge du (v_{k})) \\
& = (u^*\vol_L) (v_1, \ldots, v_k) \\
& = \lambda^k (\star \alpha)^{(0,k)} (v_1, \ldots, v_k )\\
& = \lambda^k h(P_{(0,k)} (v_1, \ldots, v_{k-1}), v_k) \\
& = \lambda^k h^{(0,2)} (P_{(0,k)} (v_1, \ldots, v_{k-1}), v_k)  \\
& = \lambda^{k-2} (u^* g) (P_{(0,k)}( v_1, \ldots, v_{k-1}), v_k) \\
& = \lambda^{k-2} g (du (P_{(0,k)} (v_1, \ldots, v_{k-1})), d u(v_k) ).
\end{align*}
Since $du_x$ is surjective, we get
$$ \star_L (du (v_1) \wedge \cdots \wedge du (v_{k-1}) ) = \lambda^{k-2} du (P_{(0,k)} (v_1, \ldots, v_{k-1}) ) $$
or equivalently
\begin{equation} \label{eq:Smith-submersion-direct-temp1}
\star_L \Lambda^{k-1} (du) = \lambda^{k-2} du \circ P_{(0,k)} \quad \text{on $\Lambda^{k-1}(T_xM)$}.
\end{equation}
From the proof of Corollary~\ref{cor:Smith-submersion-alternative-form}, we had $| (\star \alpha)^{(0,k)}| = 1$. Combining this with Lemma~\ref{lemma:PP-tr} gives
\begin{equation} \label{eq:Smith-submersion-direct-temp2}
P_{(0,k)} P_{(0,k)}^{\top} = | (\star \alpha)^{(0,k)}|^2 \, \pi^{(0,1)} = \pi^{(0,1)}.
\end{equation}
Composing with $P_{(0,k)}^{\top}$ on the right of both sides of~\eqref{eq:Smith-submersion-direct-temp1} and using~\eqref{eq:Smith-submersion-direct-temp2} and Proposition~\ref{prop:P-adjoint}, since $du \circ \pi^{(0,1)} = du$, we obtain
$$ \star_L \Lambda^{k-1}(du) (\cdot \hk (\star \alpha)^{(0,k)}) = \frac{ (-1)^{k-1}}{(\sqrt{k})^{k-2}}|du|^{k-2} du. $$
Comparing the above with~\eqref{eq:Smith-sub-P-version}, we see that it remains to verify that
$$ \Lambda^{k-1}(du) (\cdot \hk (\star \alpha)^{(0,k)}) = \Lambda^{k-1}(du) (\cdot \hk \star \alpha). $$
To see this, we take any $v \in T_x M$ and compute
\begin{align*}
& \Lambda^{k-1} (du) (v \hk (\star \alpha)^{(0,k)}) & & \\
& \qquad {} = \Lambda^{k-1 }(du) ( (v^{(1,0)} +v^{(0,1)}) \hk (\star \alpha)^{(0,k)} ) & & \\
& \qquad {} = \Lambda^{k-1} (du) (v^{(0,1)} \hk (\star \alpha)^{(0,k)}) & & \text{(because $v^{(1,0)} \hk (\star \alpha)^{(0,k)} = 0$)} \\
& \qquad {} = \Lambda^{k-1} (du) (v^{(0,1)} \hk (\star \alpha)^{(0,k)} + v^{(1,0)} \hk (\star \alpha)^{(1,k-1)}) & & \text{(because $(\star \alpha)^{(1,k-1)} = 0$ by Lemma~\ref{lemma:Smith-sub-props})} \\
& \qquad {} = \Lambda^{k-1} (du) (v \hk \star \alpha)^{(0,k-1)} & & \text{(by Lemma~\ref{lemma:hook-type})} \\
& \qquad {} = \Lambda^{k-1} (du) (v \hk \star \alpha) & & \text{(because $du$ is zero on vertical vectors)},
\end{align*}
concluding the claim.
\end{proof}

\begin{prop} \label{prop:Smith-submersion-direct-prelim2}
We have shown that if $u \colon (M^n, h,\alpha) \to (L^k, g)$ is a Smith submersion, then
\begin{equation} \label{smith-subm-eq2}
\star_L \Lambda^{k-1} (du) (\cdot \hk \star \alpha) = (-1)^{k-1} \frac{|du|^{k-2}}{\sqrt{k}^{k-2}} du.
\end{equation}
The converse also holds. That is, if~\eqref{smith-subm-eq2} holds, then $u$ is a Smith submersion.
\end{prop}
\begin{proof}
Let $x \in M$. If $du_x = 0$, which satisfies~\eqref{smith-subm-eq2} at $x$, then $u$ is a Smith submersion at $x$. Now assume $du_x \neq 0$. Let $e_1, \dots, e_m$ be an oriented orthonormal bases of $(\ker (du)_x)^{\perp}$. Note that a priori we do \emph{not} know that $m=k$. However, we have that $1 \leq m \leq k$. Let $i, j \in \{ 1, \ldots. m \}$.

We first observe that
\begin{align}
\Lambda^{k-1} (du) (e_i \hk \star \alpha) & = \Lambda^{k-1} (du) (e_i \hk \star \alpha)^{(0,k-1)} \nonumber \\
& = \Lambda^{k-1} (du) (e_i^{(0,1)} \hk \star \alpha)^{(0,k-1)} \quad \text{(because $e_i$ is already of type $(0,1)$)} \nonumber \\
& = \Lambda^{k-1} (du) (e_i^{(0,1)} \hk (\star \alpha)^{(0,k)}) \nonumber \\
& = \Lambda^{k-1} (du) (e_i \hk (\star \alpha)^{(0,k)}). \label{eqlastlinealt}
\end{align}
Evaluating both sides of~\eqref{smith-subm-eq2} on $e_i$, using~\eqref{eqlastlinealt}, and taking inner product with $du(e_j)$ we get
\begin{align*}
(-1)^{k-1} \lambda^{k-2} g(du(e_i), du(e_j))  &= g(\star_L \Lambda^{k-1} (du) (e_i \hk (\star \alpha)^{(0,k)}), du(e_j)) \\
&= u^{*} \vol_L ((e_i \hk (\star \alpha)^{(0,k)}) \wedge e_j) \\
&= u^{*} \vol_L (e_i \hk ((\star \alpha)^{(0,k)} \wedge e_j) - (-1)^k (\star \alpha)^{(0,k)} e_i \hk e_j) \\
&= u^{*} \vol_L (0 + (-1)^{k-1} \delta_{ij} (\star \alpha)^{(0,k)}) \quad \text{(because $\Omega^{(0,k+1)} = 0$)} \\
&= (-1)^{k-1} \delta_{ij} u^{*} \vol_L ((\star \alpha)^{(0,k)}) \\
&= (-1)^{k-1} \delta_{ij} u^{*} \vol_L (\star \alpha).
\end{align*}
Note that
$$ u^{*} \vol_L (\star \alpha) \, \vol_M = h(u^{*} \vol_L, \star \alpha) \, \vol_M = u^{*} \vol_L \wedge \star^2 \alpha = u^{*} \vol_L \wedge (-1)^{k(n-k)} \alpha = \alpha \wedge u^{*} \vol_L. $$
We deduce that
$$ \begin{cases} \quad \lambda^{k-2} g(du(e_i), du(e_j)) \, \vol_M = \alpha \wedge u^{*} \vol_L & \text{if $i=j$}, \\ \quad g(du(e_i), du(e_j)) = 0 & \text{if $i \neq j$}. \end{cases} $$
Using the above we compute
\begin{align*}
\alpha \wedge u^{*} \vol_L & = \frac{1}{m} \lambda^{k-2} \sum_i g(du(e_i),du(e_i)) \, \vol_M \\
& = \frac{1}{m} \lambda^{k-2} \sum_{i,j} g(du(e_i),du(e_j)) \, \vol_M \\
& = \frac{1}{m} \lambda^{k-2} |du|^2 \vol_M \\
& \geq \frac{1}{k} \lambda^{k-2} |du|^2 \vol_M \\
& = \lambda^{k} \vol_M.
\end{align*}
Combining with Theorem~\ref{thm:submersion-ineq} we get the desired equality, and thus $u$ is a Smith submersion in the sense of Definition~\ref{defn:Smith-submersion}.
\end{proof}

\begin{prop} \label{prop:Smith-immersion-direct-prelim2}
Let $P \in \Gamma(T^* M \otimes \Lambda^q (TM))$. Under the identification of vector fields with $1$-forms using the metric, assume that $P$ is totally skew-symmetric. Then $\Div(\Lambda^{q}(du)(P)) = \Lambda^{q}(du) (\Div(P))$.
\end{prop}
\begin{proof}
We trivially have equality at points where $du$ is zero. Let $x \in M^0$. Take Riemannian normal coordinates $\frac{\partial}{\partial x^i}, \frac{\partial}{\partial y^a}$ centred at $x$ and $u(x)$ respectively. For simplicity of notation, let
$$ A \coloneqq \Lambda^{q}(du) (P) \in \Gamma(T^* M \otimes \Lambda^{q}(TL)). $$
Expressing the components of $A$ and $P$ in terms of these normal coordinates at the point $x$, we compute
\begin{align*}
A^{v_1 \cdots v_{q}}_j & = \big( \Lambda^{q} (du) (P_j) \big)^{v_1 \cdots v_{q}} \\
& = \frac{1}{q!} P^{t_1 \cdots t_q}_j \left( \Lambda^{q}(du) \left( \frac{\partial}{\partial x^{t_1}} \wedge \cdots \wedge \frac{\partial}{\partial x^{t_q}} \right) \right)^{v_1 \cdots v_{q}} \\
& = \frac{1}{q!} P^{t_1 \cdots t_q}_j \left( \frac{\partial u^{s_1}}{\partial x^{t_1}} \frac{\partial}{\partial y^{s_1}} \wedge \cdots \wedge \frac{\partial u^{s_q}}{\partial x^{t_q}} \frac{\partial}{\partial y^{s_q}} \right)^{v_1 \cdots v_{q}} \\
& = P^{t_1 \cdots t_q}_j \frac{\partial u^{v_1}}{\partial x^{t_1}} \cdots \frac{\partial u^{v_q}}{\partial x^{t_q}} \qquad \text{(by skew-symmetry of $P$ in $t_1 \cdots t_q$)}.
\end{align*}
From this we obtain
\begin{align*}
(\Div A)^{v_1 \cdots v_{q}} & = (\nabla_j A)^{v_1 \cdots v_{q}}_j \\
& = (\nabla_j P)^{t_1 \cdots t_q}_j \frac{\partial u^{v_1}}{\partial x^{t_1}} \cdots \frac{\partial u^{v_{q}}}{\partial x^{t_{q}}} \\
& \qquad {} + P^{t_1 \cdots t_q}_j \sum_{\ell=1}^{q} \frac{\partial^2 u^{v_{\ell}}}{\partial x^j \partial x^{t_{\ell}}} \frac{\partial u^{v_1}}{\partial x^{t_1}} \cdots \widehat{\frac{\partial u^{v_{\ell}}}{\partial x^{t_{\ell}}}} \cdots \frac{\partial u^{v_{q}}}{\partial x^{t_{q}}},
\end{align*}
where the second term above is zero by symmetry in $j, t_{\ell}$ of $\frac{\partial^2 u^{v_{\ell}}}{\partial x^j \partial x^{t_{\ell}}}$ and skew-symmetry of $P^{t_1 \cdots t_q}_j$, by our assumption on $P$. But then the first term is just:
$$ \Lambda^{q}(du) (\Div(P))^{v_1 \cdots v_{q}}, $$
which completes the proof.
\end{proof}

\begin{thm} \label{thm:Smith-submersion-direct-harmonic}
Let $u \colon (M^n, h) \to (L^k, g)$ be a Smith submersion with respect to the calibration form $\alpha \in \Omega^{n-k}$. If $d \alpha = 0$, then $u$ is $k$-harmonic in the sense that $\tau_k (u) = 0$.
\end{thm}
\begin{proof}
By equation~\eqref{eq:p-tension} and Proposition~\ref{prop:Smith-submersion-direct-prelim}, we need to show that
$$ \Div (\star_L \Lambda^{k-1}(du) (\cdot \hk \star \alpha)) =0. $$
However, $\star_L$ commutes with $\nabla$. Moreover, the section $\cdot \hk \star \alpha \in \Gamma(T^* M \otimes \Lambda^{k-1} (T^* M))$ is totally skew-symmetric. Hence, by Proposition~\ref{prop:Smith-immersion-direct-prelim2}, it is enough to show that
$$ \Div(\cdot \hk \star \alpha) = 0. $$
But for any $\beta \in \Omega^q$ we have $\Div(\cdot \hk \beta) = -d^* \beta$, because
$$ \Div(\cdot \hk \beta)_{s_1 \cdots s_{q-1}} = \nabla_i ((\cdot \hk \beta)_i)_{s_1 \cdots s_{q-1}} = \nabla_i \beta_{i s_1 \cdots s_{q-1}} = -(d^* \beta)_{s_1 \cdots s_{q-1}}. $$
So if $d \alpha = 0$ then $\Div(\cdot \hk \star \alpha) = -d^* \star \alpha = 0$, which concludes the proof.
\end{proof}

\section{Discussion} \label{sec:discussion}

In this section we review analytic properties of Smith immersions, discuss examples of Smith immersions and Smith submersions, make some remarks on the relevance to the SYZ and GYZ conjectures of mirror symmetry involving calibrated fibrations, and present several questions for future study.

\subsection{Analytic results for Smith immersions} \label{sec:immersions-analysis}

Numerous analytic results for Smith immersions were proved in Cheng--Karigiannis--Madnick~\cite[Sections 4 and 5]{CKM1}. In that paper the authors assumed that the calibration form $\alpha \in \Omega^k (M)$ was associated to a vector cross product (VCP), but as we showed in Section~\ref{sec:Smith-immersions}, this assumption was not necessary. All the analytic results used the form~\eqref{eq:Smith-immersion-eq} of the Smith immersion equation. In this section we informally review these analytic results. (Note that when $k=2$ these analytic results concern $J$-holomorphic maps and are classical.) See~\cite{CKM1} for precise statements.

\textbf{Removable singularities.} If $u$ is a $C^1_{\mathrm{loc}}$ Smith immersion on a punctured open ball in $\R^k$ with finite $k$-energy, then $u$ extends to a $C^1$ Smith immersion across the puncture.

\textbf{Energy gap.} There exists a ``threshold energy'' $\eps_0 > 0$ such that every Smith immersion $u \colon S^{k} \to M$ with $k$-energy less than $\eps_0$ is \emph{constant}. (That is, any nontrivial solution has a \emph{minimum} $k$-energy.) This is used to show that there are only a finite number of ``bubbles''.

\textbf{Compactness modulo bubbling.} Let $W \subseteq L$ be open, and let $\{W_m\}_{m \in \N}$ an increasing sequence of open sets exhausting $W$, and $g_m$ a sequence of metrics on $W_m$ such that $g_m \to g$ in $C^{\infty}_{\mathrm{loc}}$ on $W$. Let $u_m \colon (W_m, [g_m]) \to (M, h)$ be a sequence of Smith immersions with \emph{uniformly bounded $k$-energy}.

Then there exists a Smith immersion $u_{\infty} \colon (W, g|_W) \to (M, h)$ and a (possibly empty) finite subset $\mathcal B = \{ x_1, \ldots, x_N \}$ of $L$ such that (after passing to a subsequence) the following three properties hold:
\begin{enumerate}[(a)]
\item $u_m \to u_{\infty}$ in $C^1_{\mathrm{loc}}$ on $W \setminus \mathcal B$ uniformly on compact subsets of $W \setminus \mathcal B$,
\item as Radon measures on $L$, we have $|du_m|^{k} \vol_{L} \to |du_{\infty}|^{k} \vol_{L} + \textstyle{\sum_{i=1}^N} c_i \delta(x_i)$, where $\delta(x_i)$ is a Dirac measure at $x_i$, and each $c_i \geq \frac{1}{2} \eps_0$, where $\eps_0$ is the ``threshold energy''. This says that the energy density can concentrate at points, where a minimum amount of energy is lost.
\item If the $u_m$ have uniformly bounded $p$-energy for some $p \in (k, \infty]$, then $\mathcal B = \varnothing$. (There is no bubbling.)
\end{enumerate}
(In practice we take $W = L$ or $L = S^k$ and $W = S^k \setminus \{ p^- \}$, where $p^-$ is the south pole. See~\cite[Remark 4.13]{CKM1} for details.)

This result can be applied to a sequence $u_m \colon L \to M$ of Smith immersions representing \emph{the same homology class in $H_{k} (M)$}, as they have a uniform $k$-energy bound. For each $x_i$, by rescaling about $x_i$ and using \emph{conformal invariance}, and reapplying this result, we obtain a ``bubbled off'' Smith immersion $\tilde{u}_{\infty, i} \colon S^{k} \to M$. This process stops after a \emph{finite number of iterations} due to the energy gap.

\textbf{No energy loss.} We have $\lim_{m \to \infty} E_{k} (u_m) = E_{k} (u_{\infty})  + \sum_{i} E_{k} (\tilde{u}_{\infty,i})$. This says that the limiting $k$-energy is the sum of the $k$-energy of $u_{\infty}$ plus the $k$-energy of each of the bubble maps.

\textbf{Zero neck length.} We have $u_{\infty} (x_i) = \tilde{u}_{\infty, i} (p^-)$, where $p^-$ is the south pole of $S^k$. This says that for $m > > 0$, then $u_m$ is \emph{homotopic} to the connect sum $u_{\infty} \# (\underset{i}{\#} \tilde{u}_{\infty, i})$.

It would of course be very interesting to establish analogous analytic results for Smith submersions. However, the \emph{conformal invariance} of Smith immersions, as detailed in Remark~\ref{rmk:immersion-conf-inv}, was used crucially to establish the above analytic results. By contrast, Remark~\ref{rmk:submersion-conf-inv} says that Smith submersions are only \emph{horizontally conformally invariant}. But perhaps this is indeed the right notion that is needed in this context. The authors plan to investigate this question further.

\subsection{Examples of Smith maps} \label{sec:examples}

In this section we discuss some examples of Smith maps.

\begin{ex} \label{ex:immersion}
Let $(M^n, h)$ be a Riemannian manifold equipped with a calibration form $\alpha \in \Omega^k (M)$. Let $\iota \colon L^k \to M^n$ be an immersion of an oriented manifold $L^k$ into $M$, and equip $L$ with the pullback metric $g = \iota^* h$, so that $\iota$ is a Riemannian immersion. Suppose that $\iota(L)$ is $\alpha$-calibrated, which means that $\iota^* \alpha = \vol_L$. Then $\iota$ is a Smith immersion with dilation $\lambda = 1$. Thus, any $\alpha$-calibrated submanifold gives rise to a Smith immersion, but the notion of Smith immersion is more general.

Indeed, if $f \colon (L, g) \to (L, g)$ is an orientation-preserving conformal diffeomorphism, then $u = \iota \circ f$ is also a Smith immersion, with the same image $u(L) = \iota(L)$, but $u$ need not be a Riemannian immersion.
\end{ex}

There are several examples of Smith submersions where the domain $(M^n, h)$ is noncompact, given by explicit \emph{cohomogeneity one} special holonomy metrics on total spaces $M^n$ of vector bundles over a base $L^k$, and equipped with a parallel calibration form $\alpha \in \Omega^k (M)$. These include the Bryant--Salamon examples~\cite{BS} of $\G$ and $\Spin{7}$ manifolds, and (very likely) also include the Stenzel examples~\cite{Stenzel} of Calabi--Yau metrics on $T^* S^m$. The Smith submersion is the projection map $u \colon M \to L$, and the fibres are $(n-k)$-dimensional submanifolds calibrated by $\star \alpha \in \Omega^{n-k} (M)$.

In these examples, we have $du \neq 0$ everywhere on $M$, so $M^0 = M$. (See the discussion in Section~\ref{sec:SYZ} for why we cannot expect this to happen if $M$ is compact.) We now discuss these examples in detail.

\begin{ex} \label{ex:BS-G21}
Consider the spinor bundle $M^7 = \spi(S^3)$ over the round $S^3$. There is a torsion-free $\G$-structure $\ph$ on $M^7$, with dual $4$-form $\ps = \star \ph$, inducing a metric $h$ which has holonomy $\G$. The projection $u \colon (M^7, h) \to (S^3, g)$ is a submersion. We claim that the map $u$ is a Smith submersion with respect to the calibration form $\alpha = \ps \in \Omega^4 (M)$.

To see this, we use the notation of~\cite[Section 3.1]{KL}. We have local vertical vector fields $\zeta_0, \zeta_1, \zeta_2, \zeta_3$ and horizontal vector fields $b_1, b_2, b_3$. The function $r \geq 0$ is the distance from the zero section in the fibres of $M$. Then it is known that for $c_0, c_1 > 0, \kappa > 0$ we have a torsion-free $G_2$ structure defined by
\begin{equation} \label{eq:BS-G21}
\ph = 3 \kappa(c_0 + c_1 r^2) u^* \vol_{S^3} + 4 c_1 (b_1 \wedge \Omega_1 + b_2 \wedge \Omega_2 + b_3 \wedge \Omega_3),
\end{equation}
where $\Omega_i$ are vertical $2$-forms and such that the induced metric is
$$ h = (3 \kappa)^{\frac{2}{3}} (c_0 + c_1 r^2)^{\frac{2}{3}} u^* g_{S^3} + 4 \Big( \frac{c_1^3} {3\kappa} \Big)^{\frac{1}{3}} (c_0 + c_1 r^2)^{-\frac{1}{3}} (\zeta^2_0 +\zeta^2_1 + \zeta^2_2 + \zeta^2_3). $$
Hence, we see that $h^{(0,2)} = (3 \kappa)^{\frac{2}{3}} (c_0 + c_1 r^2)^{\frac{2}{3}} u^* g_{S^3}$ which gives $u^* g_{S^3} = \lambda^2 h^{(0,2)}$ for
$$ \lambda = (3 \kappa)^{-\frac{1}{3}} (c_0 + c_1 r^2)^{-\frac{1}{3}}. $$
By Corollary~\ref{cor:Smith-submersion-alternative-form}, it remains to verify that $u^* \vol_{S^3} = \lambda^3 \ph^{(0,3)}$. But we immediately see from~\eqref{eq:BS-G21} that
$$ \ph^{(0,3)} = 3 \kappa(c_0 + c_1 r^2) u^* \vol_{S^3} = \lambda^{-3} u^* \vol_{S^3}, $$
which gives the desired equality.

Since the $\G$-structure is torsion-free, in particular we have that $d \ps = 0$. Consequently, the map $u \colon M \to S^3$ is $3$-harmonic and the fibres are calibrated by $\ps$. (That the fibres of this $\G$-manifold are coassociative submanifolds is of course well-known.)
\end{ex}

\begin{ex} \label{ex:BS-G22}
Consider the manifold $M^7 = \Lambda_{-}^2 (T^* X^4)$ of anti-self dual $2$-forms over $X$, where $X^4$ is either the round $S^4$ or the Fubini--Study $\C \PR^2$. There is a torsion-free $\G$-structure $\ph$ on $M^7$, with dual $4$-form $\ps = \star \ph$, inducing a metric $h$ which has holonomy $\G$. The projection $u \colon (M^7, h) \to (X^4, g)$ is a submersion. We claim that the map $u$ is a Smith submersion with respect to the calibration form $\alpha = \ph \in \Omega^3 (M)$.

To see this, we use the notation of~\cite[Section 4.1]{KM}. There exist positive functions $w$ and $v$ which depend only on the radial coordinate in the vertical fibres and satisfy certain differential equations such that we have a torsion-free $G_2$ structure given by
$$ \ph = v^3 \vol_{\cV} + w^2 v d \theta, $$
where $\vol_{\cV}$ is the volume form for the vertical part and $\theta$ is the canonical $2$-form on $\Lambda_{-}^2 (T^* X)$. The dual $4$-form can be expressed as
\begin{equation} \label{eq:BS-G22}
\ps = \ps^{(0,4)} + \ps^{(2,2)} \qquad \text{where \, $\ps^{(0,4)} = w^4 \, u^* \vol_X$},
\end{equation}
and the metric $h$ induced by $\ph$ is given by
$$ h = w^2 \, u^* g_{X} + v^2 g_{\cV}. $$
Hence, we see that $h^{(0,2)} = \lambda^{-2} u^* g_{X}$ for $\lambda = w^{-1}$. By Corollary~\ref{cor:Smith-submersion-alternative-form}, it remains to verify that $u^* \vol_{X} = \lambda^4 \ps^{(0,4)}$. But this is immediate from~\eqref{eq:BS-G22}.

Since the $\G$-structure is torsion-free, in particular we have that $d \ph = 0$. Consequently, the map $u \colon M \to X^4$ is $4$-harmonic and the fibres are calibrated by $\ph$. (That the fibres of this $\G$-manifold are associative submanifolds is of course well-known.)
\end{ex}

\begin{ex} \label{ex:BS-Spin7}
Consider the manifold $M^8 = \spi_- (S^4)$ of negative chirality spinors over the round $S^4$. There is a torsion-free $\Spin{7}$-structure $\Ph$ on $M^8$, inducing a metric $h$ which has holonomy $\Spin{7}$. The projection $u \colon (M^8, h) \to (S^4, g)$ is a submersion. We claim that the map $u$ is a Smith submersion with respect to the calibration form $\alpha = \Ph \in \Omega^4 (M)$.

To see this, we use the notation of~\cite[Section 4.2]{KM}. There exist positive functions $w$ and $v$ which depend only on the radial coordinate in the vertical fibres and satisfy certain differential equations such that we have a torsion-free $\Spin{7}$ structure given by
\begin{equation} \label{eq:BS-Spin7}
\Ph = w^4 \, u^* \vol_{S^4} + w^2 v^2 \beta + v^4 \vol_{\cV},
\end{equation}
where $\vol_{\cV}$ is the volume form on the vertical part and $\beta$ is some $(2,2)$-form. The metric $h$ induced by $\Ph$ is given by
$$ h = w^2 \, u^* g_{S^4} + v^2 g_{\cV}. $$
Hence, we see that $h^{(0,2)} = \lambda^{-2} u^* g_{S^4}$ for $\lambda = w^{-1}$. By Corollary~\ref{cor:Smith-submersion-alternative-form}, it remains to verify that $u^* \vol_{S^4} = \lambda^4 \Ph^{(0,4)}$. But this is immediate from~\eqref{eq:BS-Spin7}.

Since the $\Spin{7}$-structure is torsion-free, in particular we have that $d \Ph = 0$. Consequently, the map $u \colon M \to S^4$ is $4$-harmonic and the fibres are calibrated by $\Ph$. (That the fibres of this $\Spin{7}$-manifold are Cayley submanifolds is of course well-known.)
\end{ex}

\begin{ex} \label{ex:Stenzel}
There is an explicit cohomogeneity one Calabi--Yau metric $h$ on the total space of $M^{2m} = T^* (S^m)$, called the \emph{Stenzel} metric. When $m=2$ this is the classical Eguchi--Hanson metric, and when $m=3$ it is the Candelas--de la Ossa \emph{conifold metric}. (See the paper of Ionel--Min-Oo~\cite{IM} for a concrete simple description of these metrics.) Being Calabi--Yau, this Riemannian manifold $(M^{2m}, h)$ is equipped with a holomorphic complex volume form $\Upsilon \in \Omega^{(m,0)} (M)$ such that $\alpha = \real (\Upsilon) \in \Omega^m (M)$ is a special Lagrangian calibration.

Let $u \colon M^{2m} \to S^m$ be the projection. The fibres of $u$ are special Lagrangian submanifolds. It seems very likely that $u$ is a Smith submersion, so that it is horizontally conformal and an $m$-harmonic map. The authors did not explicitly verify this. At least when $m=4$, such a verification should be possible using the many useful explicit formulas in Papoulias~\cite{Papoulias}.
\end{ex}

It would be interesting to examine if other known calibrated fibrations can be described by Smith submersions. For example, Goldstein exhibits a special Lagrangian torus fibration on the Borcea--Voisin manifold in~\cite{Gold1} and other special Lagrangian fibrations in noncompact Calabi--Yau manifolds with symmetry are discussed by Gross~\cite{Gross} and Goldstein~\cite{Gold2}.

Moreover, Karigiannis--Lotay~\cite{KL} exhibit other coassociative fibrations on the Bryant--Salamon $\G$-manifold $\Lambda^2_- (S^4)$, very different from the obvious one in Example~\ref{ex:BS-G22}, and Trinca~\cite{Trinca} similarly exhibits a nontrivial Cayley fibration on the Bryant--Salamon $\Spin{7}$-manifold $\spi_- (S^4)$, very different from the obvious one in Example~\ref{ex:BS-Spin7}. Attempting to verify if these fibrations can be described by a Smith submersion seems to be an interesting but difficult problem.

\subsection{Calibrated fibrations and the SYZ and GYZ ``conjectures''} \label{sec:SYZ}

In this section we briefly discuss the potential relevance of Smith submersions to the Strominger--Yau--Zaslow~\cite{SYZ} ``conjecture'' in Calabi--Yau geometry, as well as to the analogous Gukov--Yau--Zaslow ``conjecture'' in $\G$ geometry. The authors are certainly not experts on the mathematics involved here, and we know even less about the physics. Nevertheless, we feel it worthwhile to make a few remarks. We put ``conjecture'' in quotes in both cases, as these ideas are predominantly motivated by physics, and their precise mathematical formulations are constantly evolving. Our brief discussion here is far from exhaustive, and is only meant to pique the reader's interest for further inquiry.

Roughly speaking, Strominger--Yau--Zaslow argue in~\cite{SYZ} that one should expect (at least for certain types of points near the boundary of the moduli space) that a compact Calabi--Yau complex $3$-fold should admit a fibration over a real $3$-dimensional base, necessarily with singular fibres. The generic (smooth) fibre should be a special Lagrangian torus. The mathematical inspiration comes from the deformation theory of McLean~\cite{McLean}, which shows that a compact special Lagrangian $3$-manifold $L^3$ in a Calabi--Yau $6$-manifold locally smoothly deforms in a family of dimension $b^1 (L^3)$. One then expects to construct the ``mirror Calabi--Yau manifold'' by dualizing smooth fibres and then somehow compactifying.

Similarly, Gukov--Yau--Zaslow explain in~\cite{GYZ} that, again under certain conditions, a compact torsion-free $\G$-manifold should admit a fibration over a $3$-dimensional base, again with singular fibres. The generic (smooth) fibre should be a coassocative submanifold with is topologically either $T^4$ or $K3$. Again, this is inspired by McLean's result in~\cite{McLean} that a compact coassociative $4$-manifold $L^4$ in a torsion-free $\G$-manifold locally smoothly deforms in a family of dimension $b^2_+ (L^4)$, modulo orientations.

A kew observation by Joyce~\cite{Joyce-SYZ}, discussed also in~\cite[Chapter 9]{Joyce-book}, is that special Lagrangian fibrations of compact Calabi--Yau manifolds should not be expected to be smooth generically. Rather, Joyce provides evidence that they should be \emph{piecewise-smooth}, with the singularities of the map being related to topology change of the fibres. This suggests that the set of critical fibres should be relatively large. Indeed, Joyce argues that singular fibres should generically be of codimension one. It is reasonable to believe that analogous statements should hold for coassociative fibrations of compact torsion-free $\G$-manifolds. (Baraglia~\cite{Baraglia} gives a rigorous intricate argument proving that such coassociative fibrations necessarily must have singular fibres.)

When the domain $(M, h)$ of a Smith submersion is noncompact, there exist many explicit examples of calibrated fibrations, and at least some are definitely Smith submersions, as discussed in Section~\ref{sec:examples}. However, if $(M, h)$ is compact, then we expect that there must necessarily exist singular fibres. It would be interesting to see this directly by studying the PDE~\eqref{eq:Smith-submersion-eq} satisfied by a Smith submersion.

More generally, it is crucially important to understand the size of the \emph{critical set}
$$ M^c = M \setminus M^0 = \{ x \in M : du_x = 0 \} $$
of a Smith submersion. Similarly, the critical set $L^c = L \setminus L^0 = \{ x \in L : du_x = 0 \}$ of a Smith immersion is still very mysterious. In the classical case, when $(M, h)$ is an almost K\"ahler manifold equipped with the K\"ahler calibration form $\alpha = \omega \in \Omega^2 (M)$, then a Smith immersion $u \colon (L^2, g) \to (M, h)$ with respect to $\omega$ is a $J$-holomorphic map. In this case, when $L$ is compact it is known, by methods of \emph{unique continuation}, that the critical set $L^c$ is a finite set of points. (See McDuff--Salamon~\cite[Sections 2.3--2.4]{MS} for details.) It is an important open problem to see if such methods can in any way be effectively applied to general Smith immersions and Smith submersions. Of course, we certainly do not expect the critical sets to be of dimension zero in general.

\subsection{Questions for future study} \label{sec:future}

Many questions arise naturally from our study, which are somewhat speculative. Some of these are:

\textbf{Deformation theory of Smith maps.} What is the deformation theory of a Smith map (immersion or submersion)? From Example~\ref{ex:immersion}, any calibrated submanifold gives rise to a Smith immersion. The work of McLean~\cite{McLean} studies the deformation theory of (compact) calibrated submanifolds. Interestingly, there are two kinds of behaviours. Special Lagrangian and coassociative submanifolds deform smoothly, while complex, associative, and Cayley submanifolds in general have obstructed deformations. (The second class are essentially those calibrated submanifolds whose calibration forms are associated to vector cross products, except for higher dimensional complex submanifolds.)

However, at first glance, the Smith submersion equation does not seem to see the difference between those calibrations which have smooth deformation theories and those which are obstructed (respectively called branes and instantons by Leung--Lee~\cite{LL}). Thus, it is important to reconcile the distinction in McLean's deformation theories with the existence theory of Smith submersions. For example, if the domain $(M, h)$ is compact, so that the smooth fibres of a Smith submersion are compact calibrated submanifolds, and if $\alpha$ is an associative or Cayley calibration, then we should not in general expect existence of Smith submersions with respect to $\alpha$, because associative and Cayley submanifolds are in general obstructed. (Of course, examples \emph{do occur}, such as the obvious projections from a $7$-torus or $8$-torus with their standard $\G$ or $\Spin{7}$-structures.)

It would be interesting to see if the deformation theory of Smith immersions is ``better behaved''. Note that we aways have the freedom of precomposing by an orientation-preserving conformal diffeomorphism. Such deformations should be considered in some sense trivial. We are interested in deformations of Smith immersions which are transverse to such trivial deformations. For example, start with a (compact) associative or Cayley submanifold, and describe it by a Smith immersion. Can we always deform it (nontrivially) \emph{as a Smith immersion}? This would give a class of calibrated submanifolds with a particular type of allowed singularities which nevertheless have smooth deformation spaces.

\textbf{Stability.} We have seen from the energy inequalities that Smith immersions and Smith submersions are global minimizers of the $k$-energy in a particular class of maps. Suppose that $u$ is a $k$-harmonic map, which is \emph{stable} in the sense that the second variation of the $k$-energy at $u$ is nonnegative, so $u$ is a local minimum of the $k$-energy. Under what additional assumptions on the geometry of the source and target could we ensure that such a stable $k$-harmonic map is necessarily a Smith map? The classical example of such a stability theorem is the demonstration by Siu--Yau~\cite{SY} that a stable harmonic map from $S^2 = \C \PR^1$ into a compact K\"ahler manifold $(M, h, \omega)$ with positive holomorphic bisectional curvature is necessarily $\pm$-holomorphic. Generalizing such a result should involve finding analogues of ``positive holomorphic bisectional curvature'' in Riemannian manifolds with special holonomy.

\textbf{Constructing Smith maps via flows.} If a general stability theorem as described in the previous paragraph could be established, then one could use this to attempt to construct examples of Smith immersions or Smith submersions by running the $k$-harmonic map heat flow. This is the negative gradient flow of the $k$-energy. One would have to show that (under certain assumptions on the geometries of the source and target) that the flow exists for all time and converges to a $k$-harmonic map. Then one would hope to argue that the limit must in fact be a Smith map.

\addcontentsline{toc}{section}{References}

\bibliographystyle{plain}
\bibliography{smith-submersions.bib}

\end{document}